%% file: main.tex
\documentclass[12pt]{article}%
\usepackage[left=1.1in,right=1.1in,top=1.1in,bottom=1.1in]{geometry}

\usepackage{amsmath}
\usepackage{amssymb}
\usepackage{amsthm}
\usepackage{mathtools}
\usepackage{indentfirst}
\usepackage{mathrsfs}
\usepackage{fullpage}
\usepackage{cite}
\usepackage{bm}
\usepackage{float}
\usepackage{enumerate}
\usepackage{hyperref}
\usepackage{multirow}
\usepackage{fancyhdr}
\usepackage{footmisc}
\usepackage[labelfont=bf]{caption}
\usepackage{subcaption}
\usepackage[toc,page,title]{appendix}

\DeclareMathOperator*{\diag}{diag}

\newcommand{\abs}[1]{\left\lvert#1\right\rvert}
\newcommand{\babs}[1]{\big|#1\big|}

\newcommand{\R}{\mathbb{R}}

\renewcommand{\v}[1]{\bm{#1}}
\newcommand{\vzero}{\v{0}}
\newcommand{\vones}{\v{1}}
\newcommand{\vq}{\v{q}}

\newcommand{\vQ}{\v{Q}}
\newcommand{\vv}{\v{v}}
\newcommand{\vV}{\v{V}}
\newcommand{\vx}{\v{x}}
\newcommand{\hvx}{\hat{\vx}}
\newcommand{\bvx}{\bar{\vx}}
\newcommand{\vX}{\v{X}}
\newcommand{\hvX}{\hat{\vX}}
\newcommand{\vw}{\v{w}}
\newcommand{\hvw}{\hat{\vw}}
\newcommand{\bvw}{\bar{\vw}}
\newcommand{\vW}{\v{W}}
\newcommand{\hvW}{\hat{\vW}}

\newcommand{\vf}{\v{f}}
\newcommand{\vF}{\v{F}}

\newcommand{\vH}{\v{H}}
\newcommand{\vu}{\v{u}}
\newcommand{\vU}{\v{U}}

\newcommand{\vd}{\v{d}}
\newcommand{\vp}{\v{p}}
\newcommand{\vy}{\v{y}}
\newcommand{\vtheta}{\v{\theta}}
\newcommand{\vzeta}{\v{\zeta}}

\newtheorem{theorem}{Theorem}
\newtheorem{lemma}{Lemma}
\newtheorem{prop}{Proposition}
\newtheorem{corollary}{Corollary}
\newtheorem{definition}{Definition}
\newtheorem{remark}{Remark}
\newtheorem{assumption}{Assumption}
\newtheorem{example}{Example}

\makeatletter

\newcommand{\myfnsymbol}[1]{%
	\expandafter\@myfnsymbol\csname c@#1\endcsname
}
\newcommand{\@myfnsymbol}[1]{%
	\ifcase #1
	\or \TextOrMath{\textasteriskcentered}{*}%
	\or \TextOrMath{\textdagger}{\dagger}% 
	\or 1% 
	\or 2%
	\fi
}
\newcommand{\published}{\@myfnsymbol{1}}
\newcommand{\affiliation}{\@myfnsymbol{2}}
\newcommand{\emailA}{\@myfnsymbol{3}}
\newcommand{\emailB}{\@myfnsymbol{4}}
\makeatother

\title{Energy-efficient flocking with nonlinear navigational feedback\textsuperscript{\published}}
\author{Oleksandr Dykhovychnyi\textsuperscript{\affiliation,\emailA} \qquad Alexander Panchenko\textsuperscript{\affiliation,\emailB}}
\date{}

\begin{document}

\maketitle

% Thanks notes for title uses \myfnsymbol
\renewcommand{\thefootnote}{\myfnsymbol{footnote}}
\maketitle
% Layout the \thanks notes in the order you want
\footnotetext[1]{Published in Nonlinear Dynamics. DOI: \href{https://doi.org/10.1007/s11071-024-10527-9}{10.1007/s11071-024-10527-9}}%
\footnotetext[2]{Department of Mathematics and Statistics, Washington State University, Pullman, WA 99164, USA.}%
\footnotetext[3]{o.dykhovychnyi@wsu.edu.}%
\footnotetext[4]{panchenko@wsu.edu.}%

\setcounter{footnote}{0}% Restart footnote counter
% Footnotes for rest of document uses \fnsymbol (or whatever you choose)
\renewcommand{\thefootnote}{\fnsymbol{footnote}}

\begin{abstract}
	Modeling collective motion in multi-agent systems has gained significant attention. Of particular interest are sufficient conditions for flocking dynamics. We present a generalization of the multi-agent model of Olfati--Saber \cite{olfati2006flocking} with nonlinear navigational feedback forces. Unlike the original model, ours is not generally dissipative and lacks an obvious Lyapunov function. We address this by proposing a method to prove the existence of an attractor without relying on LaSalle's principle. Other contributions are as follows. We prove that, under mild conditions, agents' velocities approach the center of mass velocity exponentially, with the distance between the center of mass and the virtual leader being bounded. In the dissipative case, we show existence of a broad class of nonlinear control forces for which the attractor does not contain periodic trajectories, which cannot be ruled out by LaSalle's principle. Finally, we conduct a computational investigation of the problem of reducing propulsion energy consumption by selecting appropriate navigational feedback forces.
\end{abstract}

\input{sections/01_introduction}

\input{sections/02_model}
\input{sections/03_dissipative}

\input{sections/04_wobblers}

\input{sections/05_non-dissipative}
\input{sections/06_computational}
\input{sections/07_optimization}
\input{sections/08_conclusion}

\section*{Acknowledgments}
\noindent The authors were partially supported by a W. M. Keck Foundation grant. 
O.D. would like to thank Vladyslav Oles for their invaluable help with the preparation of this paper.
%O.D. would like to thank Vladyslav Oles for their invaluable support in the preparation of this paper.

\begin{appendices}
\input{sections/appendix_dense_lemma_proof}
\input{sections/appendix_propeller}
\end{appendices}

\bibliographystyle{ieeetr}
\bibliography{main.bib}

\end{document}

%% file: sections/01_introduction.tex
\section{Introduction}

Swarming behavior is an essential characteristic of various biological systems, such as flocks of birds \cite{heppner1997three, hayakawa2010spatiotemporal}, schools of fish \cite{parrish2002self, hemelrijk2005density, larsson2012fish}, insect and bacteria colonies \cite{buhl2006disorder, czirok1996formation, sokolov2007concentration}.
It often enables these groups of organisms to accomplish tasks that would be impossible for an individual organism to achieve on its own.
For instance, a group of schooling fish can confuse a predator by making itself appear as a single organism \cite{larsson2012fish}, decreasing the likelihood of the group members being preyed upon.
Such natural phenomena have inspired extensive research in the development of artificial multi-agent systems that possess swarm intelligence \cite{dorri2018multi, qin2016recent}. 
One of the important examples of such systems is robotic swarms, which have recently received significant attention \cite{barca2013swarm, brambilla2013swarm, tan2013research, connor2020current, yang2021survey}.
Robotic swarms, typically consisting of relatively inexpensive and simple agents, offer high fault tolerance, cost-effectiveness, and scalability, which make them a promising tool for a broad range of tasks \cite{schranz2020swarm}.

The applications of swimming robotic swarms range from injecting nano-scale robots into a human body for delivering drugs \cite{ceylan2017mobile, bunea2020recent} to using swarms of meter-scale autonomous underwater vehicles (AUV) to perform environment monitoring tasks in open water \cite{schill2018vertex, connor2020current}.
For modeling such systems, it is crucial to take into account hydrodynamic effects of the ambient medium.
%For modeling this kind of system, it is crucial to take into account hydrodynamic effects of the medium in which a group operates to accurately represent its true dynamics.
The Dissipative Particle Dynamics (DPD) model, initially introduced in \cite{hoogerbrugge1992simulating} as a coarse-grained particle-based simulation technique for fluids, has been successfully applied to modeling collective motion in systems like nanoswimmers \cite{wang2014diffusion, xiao2015induced}, active colloidal suspensions \cite{hinz2014motility, hinz2015particle, panchenko2018spatial, barriuso2022simulating}, and red blood cells \cite{lei2012quantifying, fedosov2011predicting}.
In the DPD model, the agents are represented as point masses that experience three types of short-range pairwise forces: \textit{conservative} repulsive forces resulting from agent collisions, \textit{dissipative} damping forces modeling viscous effects of the medium, and \textit{random} forces representing stochastic effects.
In active systems like \cite{wang2014diffusion, xiao2015induced, panchenko2018spatial, hinz2014motility, hinz2015particle, barriuso2022simulating}, the agents are also capable of producing a \textit{self-propulsion} force.
In the original DPD model \cite{hoogerbrugge1992simulating}, the dissipative force, dependent on the differences in agents’ positions and velocities, is directed along the radial axis between two agents and models only the effect of extensional viscosity.
Some later extensions of the model, such as \cite{pan2008single}, \cite{junghans2008transport}, also incorporate the effects of shear viscosity by introducing an analogous force in a transverse direction, providing a more realistic representation of the viscous effects of the surrounding medium.

Terrestrial and aerial robotic swarms find their application in agriculture \cite{ball2015robotics}, rescue operations \cite{saez2010multi, hauert2009evolved}, reconnaissance for military operations \cite{dai2006prototype}, etc.
In such systems, the effects of the medium are usually assumed to be negligible and the dynamics of an agent is determined solely by its control input.
A common approach to mathematical modeling of such systems is based on the use of virtual forces \cite{bayindir2016review}.
In this approach, the control input of an agent is determined by solutions of the Newton's equations of motion, in which artificial forces, modeling agent interaction with the environment, involve measurements obtained from the agent's sensors rather than the true physical quantities.
The popular model of Olfati--Saber (O-S) \cite{olfati2006flocking} is among the ones that are based on this approach.
In this model, the velocity matching with the neighbors is achieved through the dissipative virtual force,
and the spatial formation of the group is determined by the virtual conservative force derived from an attractive/repulsive potential.
When the group is assigned an objective to follow a given target trajectory (trajectory tracking problem), a virtual leading agent can be introduced in the model.
Each agent in the group is steered to follow the virtual leader through two additional \textit{navigational feedback forces}: the conservative force, which aligns the agent's position with that of the virtual leader, and the dissipative force, which aligns the agent's velocity with that of the virtual leader.

Since the motion of agents must be organized and coherent, an important feature of collective motion is \textit{flocking}, a phenomenon when all members of the group move in the same direction with the same speed \cite{vicsek2012collective}.
For some popular models of collective motion, the conditions under which flocking occurs are well-studied \cite{vicsek1995novel, toner1998flocks, olfati2006flocking, tanner2007flocking, cucker2007emergent}.
In particular, one of the characteristics that can lead to the emergence of flocking is \textit{dissipativity}, a property of the system that its total mechanical energy is non-increasing in time \cite{brogliato2007dissipative}. 
This allows the use of mechanical energy as a Lyapunov function to show the existence of a global attractor. 
In particular, the O-S model relies on this argument to demonstrate the presence of flocking.

In the O-S model, the dynamics of the group is decoupled into the \textit{translational dynamics} of the center of mass and the \textit{structural dynamics} of the agents relative to the center of mass.
With such a decomposition, the structural dynamics does not depend on the potentially arbitrary dynamics of the virtual leader and can be shown to be dissipative and converge to the set of equilibrium solutions having zero velocities relative to the center of mass, and hence representing flocking dynamics.
However, the above decomposition relies on the fact that the navigational feedback forces are linear, which might not be flexible enough to implement an energy-efficient control.
For instance, for small deviation from the target trajectory, it might be reasonable to relax (or entirely disable) control to save the on-board energy. 
In turn, large deviations from the target trajectory might be penalized more aggressively for the same purpose.
The introduction of nonlinear navigational feedback forces, however, makes the aforementioned decoupling problematic.
Given that the coupled dynamics, in general, is not dissipative the conditions under which flocking can occur become unclear.

Another important question regarding O-S model is possibly excessive  use of limited on-board resources to maintain formation. Using potential forces for controlling formation seems natural, but it may carry a high energy cost. The question of interest is whether a reasonable formation can be maintained without attractive potential forces. 
The results obtained in the paper show that velocity alignment alone may be sufficient to ensure a good quality flock in which ambient repulsive potential forces are zero for periods of time, while the attractive control forces are not needed. Both of these factors are useful for saving on-board energy. 
The simulation results discussed in Section \ref{section:opt} show that, broadly speaking, the most energy-efficient regimes are characterized by a shorter cut-off radius for the repulsive potential forces and a longer activation radius for the position alignment forces. This allows the system to have some ``breathing room''. 
Small perturbations about an attractor state are handled by the velocity alignment forces, which saves energy. 
Our simulations imply that the OS model is likely sub-optimal. The findings are not supported by proofs and should thus be considered preliminary. 
However, the observed trend appears to be robust.

The question of relative importance of velocity alignment versus position alignment forces is clearly of interest, but its investigation in case of nonlinear navigaitonal feedback forces is complicated by lack of aforementioned decoupling.
Given that the coupled dynamics, in general, is not dissipative the conditions under which flocking can occur remain unclear.

Although numerous developments of the O-S model have been proposed since its introduction, to the best of our knowledge, this issue has not been fully addressed yet. 
Most extensions of navigational feedback forces beyond the linear case assume specific forms of nonlinearities that ultimately allow for the construction of a Lyapunov function \cite{su2011adaptive, liu2013adaptive, huang2018fixed, guo2020event, ni2022fixed, sun2019flocking, wen2015neural, wen2021optimized}.
The authors of \cite{su2011adaptive, liu2013adaptive} consider a group of agents with nonlinear intrinsic dynamics and navigational feedback gains, where the time derivatives are quadratic functions of the agents' trajectory deviations from that of the virtual leader. 
This assumption, along with a Lipschitz-like condition imposed on the intrinsic dynamic terms, ensures the existence of a Lyapunov function. In
\cite{huang2018fixed, guo2020event, ni2022fixed}, a fixed-time consensus protocol is designed by employing navigation feedback forces with polynomial navigational feedback gains.
In \cite{sun2019flocking}, more general nonlinear feedback forces are studied, but the virtual leader is required to move with a constant velocity, making the dynamics dissipative.
Some recent works also employ artificial neural networks to design control laws for flocking \cite{wen2015neural, wen2021optimized, li2023robust}. 
In particular, in the model considered in \cite{wen2015neural}, navigational feedback forces are modeled as Radial Basis Function Neural Networks (RBFNN), with a weight adaptation law that allows the construction of a Lyapunov function. A similar RBFNN-based model proposed in 
\cite{wen2021optimized} uses reinforcement learning to design an optimal weight adaptation law. Linearly parameterized neural networks to model the leader's dynamics are considered in 
\cite{li2023robust}.
While neural network-based models allow for the design of flexible control laws, they can be computationally expensive and require a high number of control parameters, which may be difficult to interpret.

In this paper, we consider the problem of trajectory tracking for micrometer- to millimeter- scale swimming robots. Our results are also applicable to flying vehicles of sub-centimeter size equipped with protective viscoelastic cage enclosures or bumpers. 
We use the DPD model to simulate interactions of the agents mediated by the ambient environment.
We assume that the agents are capable of self-propulsion, the direction and the magnitude of which are determined by means of nonlinear virtual forces analogous to the navigational feedback forces in the O-S model.
Treating the conservative and the dissipative forces in the DPD model as virtual, our model can be seen as a generalization of the O-S model with navigational feedback being nonlinear.
Although the physical nature of the two models is different, the resulting dynamics of the agents is the same.

The main contributions of the paper are as follows.

\begin{itemize}
	\item
	{\bf Global attractor:} We prove the existence of a global attractor for navigational feedback forces, which are bounded perturbations of linear ones, under the assumption that the acceleration of the virtual leader is bounded.

	\item
	{\bf Analysis of the nature of flocking:} 
	We find that the velocities of the agents converge exponentially to the center of mass velocity. 
	It may differ from the virtual leader's velocity. 
	However, the deviations are bounded and can be reduced by adjusting tunable parameters of the navigational feedback forces. % and reducing the leader's acceleration. 
	Thus, we show that agents flock to the center of mass, which remains close to the leader's trajectory.
	
	\item
	{\bf Attractor description:} %We describe the attractor for %dissipative velocity alignment forces that vanish near the origin.
	We describe the attractor for target trajectories that yield dissipative dynamics. 
	In this case, LaSalle's principle suggests that the attractor may include both equilibrium solutions and periodic rigid motions. 
	Our results exclude the existence of periodic motions for a broad class of nonlinear alignment forces. 

	\item
	{\bf Computational investigation:} 
	%\todo{This has to be updated.} 
	We explore the relative importance of velocity and position alignment forces computationally. Position-dependent potential forces include a short-range repulsive ambient forces that soften collisions and an attractive control forces that prevent agents from escaping. Setting an attractive force to zero within a sufficiently large region may yield an approximate flock in which repulsive forces vanish, with formation maintained by velocity alignment forces alone. Using numerical optimization, we find that this choice of control forces reduces the amount of energy used for propulsion. 
\end{itemize}

The class of functions that are bounded perturbations of linear ones encompasses a wide range of control protocols by allowing arbitrary nonlinearities on any bounded set containing the origin. 
Therefore, our navigational feedback forces are more flexible than those used in \cite{su2011adaptive, liu2013adaptive, huang2018fixed, guo2020event, ni2022fixed}. The only restriction we impose on the acceleration of the virtual leader is boundedness, making our model more general than that of \cite{sun2019flocking}. 
At the same time, our model is less computationally expensive than the neural network-based models proposed in \cite{wen2015neural, wen2021optimized, li2023robust} and has a small number of easily interpretable parameters.

The rest of the paper is organized as follows.
In Section \ref{section:model}, we set up our model and describe the forces that determine the dynamics of the group.
In Section \ref{section:diss}, we obtain an attractor for the case when the dynamics is dissipative due to particular choices of the target trajectories.
In Section \ref{section:wob}, we study the existence of non-equilibrium solutions in such an attractor.
In Section \ref{section:non-diss}, we consider the general case of non-dissipative dynamics and obtain conditions under which the system exhibits asymptotic flocking.
In Section \ref{section:comp}, we present numerical simulations illustrating the flocking dynamics in our system under different regimes of motion.
Finally, in Section \ref{section:opt}, we numerically solve an optimal control problem to determine configurations of the navigational feedback forces that are energy-efficient.

%% file: sections/02_model.tex
\section{Model of collective motion}
\label{section:model}

Consider a group of $N$ identical sphere-shaped agents submerged in water.
The size of an agent is assumed to be in the range 10$\mu$m-1mm.
Each agent is equipped with a miniature propeller, controlled by an on-board processor that receives information about the agent's position and velocity from sensors. 
Additionally, the agents have elastic bumpers to prevent mechanical damage during collisions.
For agents of the specified size range moving in water, the dynamics is characterized by a low Reynolds number (see, e.g., \cite{sitti2017mobile}), and therefore, viscous forces created by the ambient environment must be taken into account.

Modeling the agents as point masses, the dynamics of the group can be described by the following Newton's equations of motion:
\begin{equation}
	\label{model:dynamics}
	\begin{aligned}
		\dot{\vq}_{i} & = \vv_{i},
		\\
		M \dot{\vv}_{i} & = \sum_{j \ne i} \vf_{ij}^{C} + \sum_{j \ne i} \vf_{ij}^{D} + \vu_{i},
	\end{aligned}
	\quad	
	i = 1, \ldots, N,
\end{equation}
where $M$ is the mass of an agent, $\vq_{i}(t), \vv_{i}(t) \in \R^{d}$, $t \ge 0$, $d \in \{2, 3\}$, are the position and the velocity of the $i$-th agent, respectively, $\vf_{ij}^{C}$ and $\vf_{ij}^D$ are the conservative and the dissipative pairwise ambient forces acting on the $i$-th agent due to its proximity to the $j$-th agent, and $\vu_i$ is the self-propulsion force generated by the $i$-th agent's propeller. 
\begin{remark}
	\normalfont
	Besides the conservative and the dissipative terms, the original DPD model \cite{hoogerbrugge1992simulating} also includes terms representing stochastic effects.
	Here, we mainly consider models with high Peclet number. In such models, convective forces dominate while stochastic effects can be neglected. 
	However, the key results about the emergence of flocking, presented in Section \ref{section:non-diss}, still hold for the full DPD model provided that random forces are  bounded (see Remark \ref{non-diss:rmrk:random}).
\end{remark}

\subsection{Ambient forces}

For two agents in close proximity, the resistance of the water volume between them or repulsion from elastic bumpers induces a conservative \textit{repulsive} force:
\begin{equation*}
	\vf_{ij}^{C} \left(\vq_{ij}\right) = A w_{C} (\abs{\vq_{ij}}) \vq_{ij}, 
\end{equation*}
where $A > 0$ is a constant, $w_{C} : \R_{\ge 0} \to [0, 1]$ is a non-increasing $C^{1}$ weight function that vanishes on $[r_{C}, \infty)$ for some $r_{C} > 0$, and $\vq_{ij} = \vq_{i} - \vq_{j}$.
The value of $r_{C}$ is assumed to be close to the diameter of an agent and is referred to as the \textit{cut-off distance} of the conservative force.
The sum of all ambient conservative forces acting on the $i$-th agent can be expressed as the negative gradient of the collective repulsive potential $U(\v{Q})$, where $\vQ = ( \vq_{1}, \ldots \vq_{N})$:
\begin{equation*}
	\vf_{i}^{C} = \sum_{j \ne i} \vf_{ij}^{C} = - \nabla_{\vq_{i}} U(\v{q}).
\end{equation*}

Following \cite{czirok1996formation, romanczuk2012mean}, we assume that the dissipative force, resulting from viscous interaction between nearby agents and the ambient environment, is given by
\begin{equation*}
	\vf_{ij}^{D} \left(\vq_{ij}, \vv_{ij}\right) = - B w_{D} (\abs{\vq_{ij}}) \vv_{ij}, 
\end{equation*}
where $B > 0$ is a constant, $w_{D} : \R_{\ge 0} \to [0, 1]$ is a non-increasing $C^{1}$ function vanishing on $[r_{V}, \infty)$ for some cut-off distance $r_{V} > 0$, and $\vv_{ij} = \vv_{i} - \vv_{j}$.

Note that both ambient forces are symmetric in the sense that $\vf_{ji}^{C} = - \vf_{ij}^{C}$ and $\vf_{ji}^{D} = - \vf_{ij}^{D}$ for all $i \ne j$.
	
\subsection{Self-propulsion forces}

The common navigational objective of the group is to follow a specified \textit{target trajectory} while maintaining a flock formation. 
The target trajectory is defined in the form of the \textit{virtual leading agent}, with coordinates $\vq_{l}(t)$, $\vv_{l}(t) \in \mathbb{R}^{d}$, $t \ge 0$, $l = N + 1$, updated as
\begin{equation}
	\label{model:leader}
	\dot{\vq}_{l} = \vv_{l}, \quad \dot{\vv}_{l} = \vf^{L} (t),
\end{equation}
where $\vf^{L} : [0, \infty) \to \R_{\ge 0}$ is $C^{1}$. 
The nonlinear self-propulsion force produced by an agent is a navigational feedback force that aligns the agent's coordinates with those of the virtual leader.
The force is given by 
\begin{equation}
	\label{model:sp-force}
	\vu_{i} = \vu^{P}(\vq_{il}) + \vu^{V}(\vv_{il}).
\end{equation}
The first term in \eqref{model:sp-force} penalizes for an agent's position deviation from the position of the virtual leader, and is defined by
\begin{equation}
	\label{model:sp-force-pos}
	\vu^{P}(\vq_{il}) =  - h \left( \abs{\vq_{il}} \right) \vq_{il}.
\end{equation}
We refer to \eqref{model:sp-force-pos} as the \textit{position alignment force}.
The second term in \eqref{model:sp-force} penalizes for an agent's velocity deviation from that of the virtual leader, and is defined by
\begin{equation}
	\label{model:sp-force-vel}
	\vu^{V}(\vv_{il}) =  - g \left( \abs{\vv_{il}} \right) \vv_{il}.
\end{equation}
\eqref{model:sp-force-vel} is referred to as the \textit{velocity alignment force}.
The functions $h, g$, referred to as \textit{generating functions} of the corresponding forces, are given by
\begin{equation}
	\label{model:eq:h-s-def}
	h(s) =
	\begin{cases}
		0, & 0 \le s \le r_{0}, \\
		 \alpha \frac{k(s)}{s}, & s > r_{0},
	\end{cases},
	\quad 
	g(s) =
	\begin{cases}
		0, & 0 \le s \le v_{0}, \\
		\beta \frac{p(s)}{s}, & s > v_{0},
	\end{cases}
\end{equation}
where $\alpha, r_{0}, \beta, v_{0} \ge 0$ are tunable control parameters, and $k : (r_{0}, \infty) \to \R_{> 0}$ and $p : (v_{0}, \infty) \to \R_{> 0}$ are increasing functions such that $h$ and $g$ are $C^{1}$.
\begin{remark}
	\normalfont
	Adjusting $\alpha$ and $\beta$ in \eqref{model:eq:h-s-def} amplifies the effects of the position alignment and the velocity alignment forces, respectively.
	The parameters $r_{0}$ and $v_{0}$ allow for ``relaxing'' the control by turning the propulsion off when the deviations from the target trajectory are small.
	In particular, when $r_{0} > 0$, the agents are forced to stay within the ball of radius $r_{0}$ centered at $\vq_{l}(t)$.
	If the group is sufficiently large and $r_{0}$ is small enough for the ambient dissipative force to act on the agents, its damping effect can be leveraged to align agents' velocities.
	In this way, the agents' on-board energy consumption can be reduced.
	A detailed discussion of the effect of parameters $r_{0}$ and $v_{0}$ on the agents' dynamics is presented in Sections \ref{section:comp} and \ref{section:opt}.
\end{remark}

The position alignment force is conservative:
\begin{equation*}
	- h \left( \abs{\vq_{il}} \right) \vq_{il} = - \nabla_{\vq_{i}} \Phi(\vQ, \vq_{l}),	
\end{equation*}
where $\Phi$ is a virtual attractive potential.
In addition, it follows from the definition of $h(s)$ that $\Phi$ is at least $C^{1}$ and satisfies
\begin{equation}
	\label{model:eq:ubounded-Phi}
	\lim_{\abs{\vq_{il}} \to \infty} \Phi(\vQ, \vq_{l}) = \infty,
	\;\;
	i = 1, \ldots, N,
\end{equation}
whenever $\alpha > 0$.

\begin{remark}
	\normalfont
	The model presented above is different from the O-S model in the following two aspects. 
	First, our navigational feedback forces are, in general, nonlinear, encompassing the linear forces of the O-S model as a special case when $k(s) = s$, $p(s) = s$, $r_{0} = v_{0} = 0$ in \eqref{model:eq:h-s-def}.
	Second, while our conservative force is purely repulsive, the conservative force in the O-S model is repulsive only at small distances and becomes attractive at larger ones.
	Although this distinction does not technically allow us to claim that our model is more general than the O-S model, these characteristics of the conservative force are not relevant for establishing the existence of flocking in either \cite{olfati2006flocking} or our further discussion.
	The fact that the short-range conservative force is both repulsive and attractive is used in \cite{olfati2006flocking} to ensure a certain spatial structure of the flock ($\alpha$-lattice).
	However, the overall cohesion of the group is rather guaranteed by the position-dependent navigational feedback force.
\end{remark}

\subsection{Proximity graph and vector form of the dynamics}
Define the \textit{proximity graph} of the group as an undirected weighted graph $\mathcal{G}(\v{Q}) = \left( \mathcal{A}, \mathcal{E}(\vQ), \sigma \right)$, with $\mathcal{A} = \{1 , \ldots , N\}$, $\mathcal{E}(\vQ) = \{ (i, j) \in \mathcal{A} \times \mathcal{A} : \abs{\vq_{ij}} \le r_{D} \}$, and $\sigma(i, j) = B w_{D} (\abs{\vq_{ij}})$.
Then the sum of all ambient dissipative forces acting on the $i$-the agent can be written as $\sum_{j \ne i} \vf_{ij}^{D} = d_{i} \vv_{i} - \sum_{j \ne i} a_{ij} \vv_{j}$, where $d_{i}$ is the degree of the $i$-th vertex of $\mathcal{G}$ and $a_{ij}$ is the $(i, j)$-th element of its adjacency matrix.
Letting $\mathcal{L}(\vQ)$ denote the Laplacian of $\mathcal{G}(\vQ)$ and $L (\vQ) = \mathcal{L} (\vQ) \otimes I_{d}$, where $I_{d}$ is the $d \times d$ identity matrix, we can write the governing equations (\ref{model:dynamics}) in the following vector form:
\begin{equation}
	\begin{aligned}
		\dot{\vQ} & = \vV,
		\\
		M \dot{\vV} & = - \nabla_{\vQ} U(\vQ) - L(\vQ) \vV - \nabla_{\vQ} \Phi(\vQ, \vq_{l}) - G(\vV - \vones_{N} \otimes \vv_{l}) (\vV - \vones_{N} \otimes \vv_{l}),
	\end{aligned}
	\label{model:dynamics-vec}
\end{equation}
where $\vV = (\vv_{1}, \ldots \vv_{N})$, $G(\vV - \vones_{N} \otimes \vv_{l}) = \diag \left(g(\abs{\vv_{1l}}), \ldots,  g(\abs{\vv_{Nl}}) \right) \otimes I_{d}$ and $\v{1}_{N}$ is an $N$-dimensional vector of all ones.
Note that the right-hand side of (\ref{model:dynamics-vec}) depends on $\vq_{l}(t)$ and $ \vv_{l}(t)$, so, in general, the system (\ref{model:dynamics-vec}) is non-autonomous.

%% file: sections/03_dissipative.tex
\section{Dissipative dynamics}
\label{section:diss}

In this section, our goal is to describe the asymptotic behavior of the system for the case when the acceleration $\vf^{L}$ of the virtual leader is identically zero, so that the navigational objective becomes a uniform motion.
Such target trajectories are among the simplest ones, yet they are typical for many scenarios.
Furthermore, a wide range of more complex trajectories can be represented by a sequence of such simple trajectories.
We are interested in determining whether the system exhibits flocking which we formally define as follows.
\begin{definition}
	\label{diss:def:flocking}
	\begin{enumerate}[(a)]
		\item The group of agents is said to exhibit \textbf{(approximate) flocking} if there exist $R, V > 0$ and $T \ge 0$ such that
		\begin{subequations}
			\begin{align}
				\label{diss:eq:flocking-def-appr-pos}
				& \abs{\vq_{i}(t) - \vq_{l}(t)} \le R,
				\\
				\label{diss:eq:flocking-def-appr-vel}
				& \abs{\vv_{i}(t) - \vv_{l}(t)} \le V, 
			\end{align}
		\end{subequations}
		for $i, j = 1, \ldots, N$ and all $t \ge T$.
		
		\item If the group exhibits flocking and
		\begin{equation}
			\label{diss:eq:flocking-def-exact-vel}
			\abs{\vv_{i}(t) - \vv_{j}(t)} \to 0, 
		\end{equation}
		for $i, j = 1, \ldots, N$, we say that the flocking is \textbf{exact}.
		
		\item If the group exhibits exact flocking and
		\begin{equation}
			\label{diss:eq:flocking-def-prop-vel}
			\abs{\vv_{i}(t) - \vv_{l}(t)} \to 0, 
		\end{equation}
		for $i = 1, \ldots, N$, we say that the exact flocking is \textbf{proper}.
	\end{enumerate}
\end{definition}
Note that part (b) of the above definition requires the agents to reach asymptotic velocity alignment, allowing the consensus value to be distinct from the velocity of the virtual leader. 
Meanwhile, in part (c), the agents must asymptotically match their velocities with that of the virtual leader, and hence part (c) defines the most desirable scenario.

In the forthcoming theoretical analysis, we assume that the control parameters $\alpha, \beta > 0$, so that both the position alignment and the velocity alignment forces are present.
If $\vf^{L}(t) \equiv \vzero$, the system becomes purely dissipative, and the standard Lyapunov-like argument, such as the one used in  \cite{olfati2006flocking}, can be applied to describe its asymptotic behavior.

For further development, it is convenient to introduce a coordinate system centered at the virtual leader's coordinates (moving frame coordinates).
Let 
\begin{equation*}
	\vx_{i} = \vq_{i} - \vq_{l},
	\;\;
	\vw_{i} = \vv_{i} - \vv_{l},
	\;\;
	i = 1, \ldots N,
\end{equation*}
and $\vX = (\vx_{1}, \ldots \vx_{N})$, $\vW = (\vw_{1}, \ldots \vw_{N})$.
Since $\vx_{i} - \vx_{j} = \vq_{i} - \vq_{j}$ and $\vw_{i} - \vw_{j} = \vv_{i} - \vv_{j}$ for $i, j = 1, \ldots, N$, the governing equations can be written as 
\begin{equation}
	\label{diss:dynamics}
	\begin{aligned}
		\dot{\vX} & = \vW,
		\\
		M \dot{\vW} & = - \nabla_{\vX} U(\vX) - L(\vX) \vW - \nabla_{\vX} \Phi(\vX) - G(\vW) \vW - M \vf^{L} \otimes \vones_{N}.
	\end{aligned}
\end{equation}
In the above coordinate system, the group will exhibit proper exact flocking if $\abs{\vW(t)} \to \vzero$ as $t \to \infty$ and $\abs{\v{X}(t)} \le R$, $t \ge T$, for some $R, T \ge 0$. 
If $\abs{\vW(t)}$ merely stays bounded for $t \ge T$, then the flocking would only be approximate.

Let 
\begin{equation}
	\label{diss:eq:energy-def}
	E (\vX, \vW) = \frac{1}{2} M \abs{\vW}^{2} + U(\vX)+ \Phi(\v{X})
\end{equation}
be the total mechanical energy of (\ref{diss:dynamics}). 
Then
\begin{equation}
	\label{diss:energy-evol}
	\frac{d}{dt} E (\vX(t), \vW(t))
	= - \vW^{T} \left[ L(\vX) +  G(\vW) \right] \vW - M \vf^{L} \otimes \vones_{N} \cdot \vW.
\end{equation}
If the right-hand side of \eqref{diss:energy-evol} is non-positive, the dynamics of (\ref{diss:dynamics}) is dissipative.
In particular, this will be the case if $\vf^{L}(t) \equiv \vzero$ since both $G(\vW)$ and $L(\vX)$ are positive semi-definite (see, e.g., \cite{godsil2001algebraic}).
\begin{remark}
	\normalfont
	Clearly, the system would also be dissipative if either $\vf^{L} \otimes \vones_{N} \cdot \vW > 0$ for all $ t \ge 0$, or if
	\begin{equation}
		\label{diss:eq:other-dissipation}
		\vf^{L} \otimes \vones_{N} \cdot \vW < 0
		\;\;
		\text{and}
		\;\;
		M \abs{\vf^{L} \otimes \vones_{N} \cdot \vW} \le \vW^{T} \left[  L(\vX) + G(\vW) \right] \vW,
	\end{equation}
	for all $t \ge 0$.
	However, it is unclear how to rigorously and explicitly characterize the target trajectories that satisfy either of the conditions.
	To better understand the scenario when $\vf^{L} \otimes \vones_{N} \cdot \vW > 0$, observe that $\vf^{L} \otimes \vones_{N} \cdot \vW = N  \vf^{L} \cdot \bar{\vw}$, where $\bar{ \vw} = \frac{1}{N} \sum_{i} (\vv_{i} - \vv_{l})$ is the average deviation of the agents' velocities from that of the virtual leader. 
	This implies that $\vf^{L} \otimes \vones_{N} \cdot \vW$ is positive if the angle between $\vf^{L}$ and $\bar{\vw}$ is less than $\pi / 2$.
	Informally speaking, the condition will hold when the virtual leader does not steer the group significantly away from the direction it is currently heading, or, in other words, when the target trajectory is sufficiently smooth. 
	However, it is unclear how to quantify such a requirement.
	For \eqref{diss:eq:other-dissipation} to hold, the dissipation of energy resulting from the action of the ambient force $\vf^{D}$ should, in a way, dominate the influx of energy in the system generated by the acceleration of the virtual leader.
	In other words, the trajectory is allowed to be more rough than it should be for $\vf^{L} \otimes \vones_{N} \cdot \vW > 0$ to hold, however, its roughness should be compensated by the damping effect of the dissipative forces.
	Still, it is unclear what exact conditions $\vf^{L}$ should satisfy for this to be true.
\end{remark}
\begin{remark}
	\normalfont
	Using the above considerations, one can describe a specific scenario in which the total mechanical energy increases, causing the dynamics to no longer be dissipative. 
	Namely, consider a group of agents moving with identical velocities $\vw_{i} = \vw$, $i = 1, \ldots, N$, for some $\vw \in \R^{d}$.
	Suppose that the agents must be stopped at time $t_{0} \ge 0$ by applying $\vf^{L} (t) = - \lambda \vw (t)$, where $t \ge t_{0}$ and $\lambda > 0$. 
	Then
	\begin{equation*}
		\frac{d}{dt} E (\vX(t), \vW(t)) = \left[- g(\abs{\vw}) + M \lambda \right] N \abs{\vw}^{2} > 0
 	\end{equation*}
 	for $t \ge t_{0}$ if $\lambda$ is sufficiently large. 
\end{remark}
\begin{remark}
	\normalfont
	In \cite{olfati2006flocking}, the origin of the moving frame coordinate system is chosen to be the center of mass $\left( \bar{\vq}, \bar{\vv} \right) = \left( \frac{1}{N} \sum_{i = 1}^{N} \vq_{i},  \frac{1}{N} \sum_{i = 1}^{N} \vv_{i}  \right)$ of the group. 
	If the generating functions $g(s)$ and $h(s)$ of the self-propulsion forces are constant (so that the forces become linear), in such a coordinate system, the structural dynamics of the group can be decoupled from that of the virtual leader (\cite{olfati2006flocking}, Lemma 2). 
	As a result, (\ref{diss:dynamics}) becomes autonomous.
	Furthermore, the first term in the right-hand side of (\ref{diss:energy-evol}) disappears, so the system is dissipative even when $\vf^{L}(t) \not\equiv \vzero$.
	However, since we do not assume that the self-propulsion forces are linear, such an approach is not useful.
	Furthermore, analyzing the structural dynamics of the group relative to the center of mass leaves behind the question of how the center of mass moves relative to the target trajectory (see, e.g., \cite{su2009flocking}).
\end{remark}

For now, suppose that $\vf^{L}(t) \equiv \vzero$.
In this case, if either $L(\vX)$ or $G(\vW)$ is positive definite for all $t \ge 0$, the system will dissipate energy whenever $\vW \ne \vzero$.
In particular, this would be true if the activation threshold $v_{0} = 0$. 
In this case, the generating function $g(s)$ of the velocity alignment force does not vanish anywhere except zero, and hence $G(\vW)$ becomes positive definite.
We refer to such a scenario as \textit{strictly dissipative}.
In turn, if $v_{0} > 0$, the diagonal elements of $G(\vW)$ can simultaneously vanish for some $\vW \ne \vzero$.
Such a scenario is referred to as \textit{non-strictly dissipative}.

For both of the above scenarios, $E (\vX, \vW)$ becomes a weak Lyapunov function for the autonomous system (\ref{diss:dynamics}).
Fix some initial conditions $(\vX(0), \vW(0)) = (\vX_{0}, \vW_{0})$, and let
\begin{equation*}
	\mathcal{U} = \{(\v{X}, \v{W}) \in \R^{2Nd} : E (\vX, \vW) \le  E (\vX_{0}, \vW_{0}) \},
\end{equation*}
Note that, for any $(\v{X}, \v{W}) \in \mathcal{U}$, we have 
\begin{equation*}
	\abs{\vW}^{2} \le \frac{2}{M} E(X, W)  \le \frac{2}{M}  E (\vX_{0}, \vW_{0}),
	\quad
	\Phi(\vX) \le E(X, W) \le  E (\vX_{0}, \vW_{0}).
\end{equation*}
and therefore $\mathcal{U}$ is bounded, due to \eqref{model:eq:ubounded-Phi}.
Thus, $\mathcal{U}$ is a compact set that is also positively invariant under the flow generated by (\ref{diss:dynamics}).
By LaSalle's invariance principle \cite{lasalle1960some}, every solution starting in $\mathcal{U}$ converges to the largest positively invariant subset of the set 
\begin{equation*}
	\mathcal{M} = \{ (\v{X}, \v{W}) \in \mathcal{U} : \frac{d}{dt} E \left( \vX(t), \vW(t) \right) = 0\}.
\end{equation*}
Therefore, any solution of (\ref{diss:dynamics}) converges to a positively invariant set $\mathcal{A}$ consisting of orbits satisfying
\begin{equation}
	\label{diss:attractor-eq-vec}
	\vW^{T} \left[ L(\vX) + G(\vW) \right] \vW = 0,
\end{equation}
or, equivalently,
\begin{equation}
	\label{diss:attractor-eq-scal}
	B \sum_{(i, j) \in \mathcal{E}(\vX)} w_{D} (\abs{\vx_{ij}})(\vw_{i} - \vw_{j})^{2} 
	+ \sum_{i = 1}^{N} g(\abs{\vw_{i}}) \abs{\vw_{i}}^{2} = 0,
\end{equation}
where the first sum is the expansion of the quadratic form $\vW^{T} L(\vX) \vW$ (see, e.g., \cite{godsil2001algebraic}).

Let $c(t)$ denote the number of connected components of the proximity graph at time $t \ge 0$, and $N_{i}(t)$, $i = 1, \ldots, c(t)$, be the number of agents in the $i$-th connected component at time $t$.
It follows from \eqref{diss:attractor-eq-scal} that, for any solution $(\vX^{\star}, \vW^{\star})$ satisfying $(\vX^{\star}(t_{0}), \vW^{\star}(t_{0})) \in \mathcal{A}$ for some $t_{0} \ge 0$, we have
\begin{equation}
	\label{diss:attractor-w}
	\vW^{\star}(t) = (\vw^{\star}_{1}(t) \otimes \vones_{N_{1}(t_{0})}, \dots, \vw^{\star}_{c(t_{0})}(t) \otimes \vones_{N_{c(t_{0})}(t_{0})}),
\end{equation}
where $\vw_{1}^{\star}(t), \ldots, \vw_{c(t_{0})}^{\star}(t) \in \R^{d} \times [t_0, \infty)$ are such that $\abs{\vw_{i}^{\star}(t)} \leq v_{0}$ for $i = 1, \ldots, c(t_{0})$ and $ t\in [t_{0}, \infty)$.
In \eqref{diss:attractor-w}, we assigned consequent indices to agents contained in a given connected component of the proximity graph, which can be done without loss of generality.
Note that \eqref{diss:attractor-w} simply means that all agents in a given connected component of the proximity graph will have the same velocity for all $t \ge t_{0}$.

Thus, for the strictly dissipative case, we arrive at the following result.
\begin{theorem}
	\label{diss:thm:flocking}
	Suppose that $\vf^{L}(t) \equiv \vzero$. Then the group of agents exhibits proper exact flocking whenever $v_{0} = 0$.
\end{theorem}
\begin{proof}
	\eqref{diss:eq:flocking-def-appr-pos} in Definition \ref{diss:def:flocking} follows from the fact that $\mathcal{U}$ is bounded for any given initial conditions.
	In turn, the assumption regarding the control parameter $v_{0}$ implies that $\vW^{\star}(t) \equiv \vzero$, and \eqref{diss:eq:flocking-def-prop-vel} in Definition \ref{diss:def:flocking} follows.
\end{proof}

It is clear that, in the strictly dissipative case, every element of $\mathcal{A}$ is an equilibrium solution. 
Furthermore, any equilibrium solution of \eqref{diss:dynamics} is a positively invariant set satisfying \eqref{diss:attractor-eq-vec}, and thus should be contained in $\mathcal{A}$.
Hence, in the strictly dissipative case, $\mathcal{A}$ is precisely the set of equilibrium solutions of \eqref{diss:dynamics}, and any solution will converge to a flock whose spatial configuration remains \textit{static} relative to the position of the virtual leader.
As it is shown in the next section, this might not be the case for the non-strictly dissipative scenario.

%% file: sections/04_wobblers.tex
\section{Wobblers}
\label{section:wob}

In this section, our goal is to determine if, in the non-strictly dissipative case, the attractor $\mathcal{A}$ may contain non-equilibrium solutions.
We will refer to such non-equilibrium solutions as \textit{wobblers}. 
First, we obtain sufficient conditions that can guarantee the non-existence of wobblers.
If these conditions holds, the qualitative behavior of the system is identical to that in the strictly dissipative case, and the group exhibits proper exact flocking.
Next, we prove the existence of wobblers for certain choices of the self-propulsion forces by constructing them explicitly.
In this way, we demonstrate that, for the case of uniform motion, the asymptotic flocks may not be static relative to the position of the virtual leader, and the exact flocking may not be proper.

Suppose that $(\vX^{\star}, \vW^{\star})$ is such that $(\vX^{\star}(t_{0}), \vW^{\star}(t_{0})) \in \mathcal{A}$ for some $t_{0} \ge 0$.
For simplicity, consider the case when the proximity graph is connected at $t_{0}$, so that, for all $t \ge t_{0}$,
\begin{equation}
	\label{wob:eq:w}
	\vW(t) = (\vw^{\star}(t), \vw^{\star}(t), \dots, \vw^{\star}(t))
\end{equation}
for some $\vw^{\star}(t) \in \R^{d} \times [t_0, \infty)$ satisfying 
\begin{equation}
	\label{wob:eq:w-bounded}
	\abs{\vw^{\star}(t)} \leq v_{0}.
\end{equation}
If $c(t_{0}) > 1$, the forthcoming analysis can be applied to each of the individual connected components of the proximity graph.

The trajectory $(\vX^{\star}, \vW^{\star})$ satisfying \eqref{wob:eq:w} solves
\begin{eqnarray}
	\dot{\vX}^{\star} & = & \vW^{\star},
	\nonumber
	\\
	\dot{\vw}^{\star} & = &\vf^C_{i} + \vu(\vx_{i}^{\star}), 
	\;\;
	i = 1, \ldots, N,
	\label{wob:eq:wobbler}
\end{eqnarray}
where $\vu(\vx_{i}^{\star}) = - h(\vx_{i}^{\star}) \vx_{i}^{\star}$.
Let $\vx_{i}^{0} = \vx_{i}(t_{0})$, $i = 1, \ldots, N$, and write
\begin{equation}
	\label{wob:eq:x-shift}
	\vx_{i}^{\star}(t)=\vx_{i}^{0}+\vx^{\star}(t),
\end{equation}
where
\begin{equation*}
	\vx^{\star}(t)=\int_{0}^{t} \vw^{\star}(\tau) d\tau.
\end{equation*}
Since conservative forces $\vf_{i}^C$ are functions of relative positions, \eqref{wob:eq:x-shift} implies that the differences
\begin{equation*}
	\vu(\vx_{i}^{\star})-\vu(\vx_{j}^{\star})
\end{equation*}
are independent of $t$. Therefore,
\begin{equation}
	\label{wob:eq:rigidity}
	\vu(\vx_{i}^{0}+\vx^{\star}(t))-\vu(\vx_{j}^{0}+\vx^{\star}(t))=\vu(\vx_{i}^{0})-\vu(\vx_{j}^{0})
\end{equation}
for all $t\in [0, \infty)$. 
Equation \eqref{wob:eq:rigidity} implies the existence of $\vd(t): [0, \infty) \to \R^{d}$ such that
\begin{equation}
	\label{wob:eq:u-shift}
	\vu(\vx_{i}^{0}+\vx^{\star}(t))=\vd(t)+\vu(\vx_{i}^{0})
\end{equation}
for $i=1, \ldots, N$.

It is natural to require that no agents can have the same initial position:
\begin{equation}
	\label{wob:eq:distinct}
	\vx_{i}^{0}\ne \vx_{j}^{0}\;\;{\rm if}\;i\ne j
\end{equation}
for $i, j = 1, \ldots, N$.
Also, due to \eqref{model:eq:ubounded-Phi}, any orbit in $\mathcal{A}$ should satisfy
\begin{equation}
	\label{wob:eq:x-bounded}
	|\vx_{i}^{0}+\vx^{\star}(t)| \leq R
\end{equation}
for $i=1, \ldots, N$ and all $t\in[t_{0}, \infty)$, where $R > 0$ is a constant.
Finally, we assume that $\vw^{\star}$ does not vanish on intervals:
\begin{equation}
	\label{wob:eq:dnvi}
	\forall
	t_1, t_2 \in {\mathbb R},   \;\; t_0 \le t_1<t_2 \Rightarrow  \vw^{\star}(t)  \not\equiv 0 \;{\rm on}\;(t_1, t_2).
\end{equation}
If there exist $t_1, t_2 \in \R,  t_0 \le t_1 < t_2$, such that $\vw^{\star}(t) \equiv 0$ on $(t_1, t_2)$, it follows that $\vw^{\star}(t) \equiv 0$ on $(t_{1}, \infty)$.
This makes $(\vX^{\star}, \vW^{\star})$ an equilibrium solution, which is not of interest in this section.
\begin{definition}
	A wobbler is a solution of \eqref{wob:eq:wobbler} satisfying \eqref{wob:eq:w-bounded}, \eqref{wob:eq:distinct}, \eqref{wob:eq:x-bounded}, and \eqref{wob:eq:dnvi}.
\end{definition}

\subsection{Non-linear case: non-existence}

First, suppose that the generating function $h(s)$ is nonlinear, i.e., either that $r_{0} > 0$ or that $k(s)$ is not a multiple of $s$.
In this case, we can eliminate the existence of wobblers for a certain class of functions $h(s)$ when the group consists of \textit{more than two agents}.
To do so, we need a few auxiliary results.
\begin{lemma}
	\label{wob:lemma:dense}
	Let $f : \R \to \R_{\ge 0}$ be continuous and $S = \{ x \in \R \mid f(x) = 0\}$.
	For any open $I \subseteq \R$, if $I \cap S$ is dense in $I$, then $f(x) = 0$ on $I$.
\end{lemma}
\begin{proof}
	See Appendix \ref{appendix:dense_lemma_proof}.
\end{proof}
Note that it follows from \eqref{model:eq:h-s-def} that $\vu(x)$ is one-to-one outside of the ball $B(\vzero, r_{0})$ since $k(s)$ is strictly increasing.
\begin{prop}
	\label{wob:prop:wobblers}
	Suppose that $(\vX^{\star}, \vW^{\star})$ is a wobbler.
	Then there exists an open interval $I \subseteq (t_{0}, \infty)$ such that 
	\begin{enumerate}[(i)]
		\item $\abs{\vx_{i}^{\star}(t)} > r_{0}$ on $I$ for $i = 1, \ldots, N$;  
		\item $\vd(t)$ cannot be constant on any open subinterval of $I$.
	\end{enumerate}
\end{prop}
\begin{proof}
	Define
	\begin{equation*}
		\mathcal{I}_{-}(t) = \{ i \mid \abs{\vx_{i}^{\star}(t)} < r_{0} \},
		\;\;
		\mathcal{I}_{0}(t) = \{ i \mid \abs{\vx_{i}^{\star}(t)} = r_{0} \},
		\;\;
		\mathcal{I}_{+}(t) = \{ i \mid \abs{\vx_{i}^{\star}(t)} > r_{0} \},
	\end{equation*}
	for all $t \ge t_{0}$
	
	Suppose that (i) is false and let $I \subseteq (t_{0}, \infty)$ be open.
	Then there exits $\tau \in I$ such that $\mathcal{I}_{0}(\tau) \cup \mathcal{I}_{-}(\tau)$ is non-empty.
	Suppose that $\mathcal{I}_{+}(\tau)$ is also non-empty.
	By continuity of $\vx^{\star}$, $i \in \mathcal{I}_{-}(\tau) \cup \mathcal{I}_{+}(\tau)$ implies that there exists $\delta_{i} > 0$ such that either $i \in \mathcal{I}_{-}(t)$ for all $t \in (\tau - \delta_{i}, \tau + \delta_{i})$, or $i \in \mathcal{I}_{+}(t)$ for all $t \in (\tau - \delta_{i}, \tau + \delta_{i})$.
	Thus,
	\begin{equation}
		\label{wob:eq:i-constant}
		\mathcal{I}_{-}(t) = \mathcal{I}_{-}(\tau),
		\;\;
		\mathcal{I}_{+}(t) = \mathcal{I}_{+}(\tau)
	\end{equation}
	on $\tilde{I} = (\tau - \delta, \tau + \delta)$, where $\delta = \min_{i \in \mathcal{I}_{-}(\tau) \cup \mathcal{I}_{+}(\tau)} \delta_{i}$.
	Then, for any $i \in \mathcal{I}_{-}(\tau)$ and any $j \in \mathcal{I}_{+}(\tau)$, subtracting the $j$-th equation from the $i$-th equation in \eqref{wob:eq:wobbler} yields
	\begin{equation*}
		\vu \left(\vx_{j}^{\star}(t)\right) = \vf_{i}^{C} - \vf_{j}^{C}
	\end{equation*}
	for all $t \in \tilde{I}$.
	\eqref{wob:eq:x-shift} implies that $\vx_{i}(t) - \vx_{j}(t)= \vx_{i}^{0} - \vx_{j}^{0}$ for all $i, j \in \{1, \ldots, N\}$ and all $t \ge t_{0}$, and thus $\vf_{i}^{C} - \vf_{j}^{C}$ is constant.
	Therefore, $\vu \left(\vx_{j}^{\star}(t)\right)$ is constant on $\tilde{I}$, and so is $\vx_{j}^{\star}(t)$ since $\vu$ is one-to-one.
	Then $\vw^{\star}(t) = 0$ on $\tilde{I}$, which contradicts the definition of a wobbler, and therefore $\mathcal{I}_{-}(\tau)$ must be empty. 
	An analogous argument can be used to show that there is no open interval containing $\tau$ such that $i \in \mathcal{I}_{0}(t)$ on it for some $i \in  \mathcal{I}_{0}(\tau)$, unless $\mathcal{I}_{+}(\tau)$ is empty.
	
	Let $i_{1} \in \mathcal{I}_{0}(\tau)$.
	If the set $S_{1} = \{ t \in \tilde{I} \mid i_{1} \in \mathcal{I}_{0}(t) \} $ is dense in $\tilde{I}$, then, by Lemma \ref{wob:lemma:dense}, $i_{1} \in \mathcal{I}_{0}(t)$ on $\tilde{I}$, which is a contradiction.
	Therefore, there exists an open interval $\tilde{I}_{1} \subseteq I$ such that $i_{1} \in  \mathcal{I}_{0}(t)$ for all $t \in \tilde{I}_{1}$.
	Similarly, if $i_{2} \in \mathcal{I}_{0}(\tau) \setminus \{ i_{1} \}$, then there exists an open interval $\tilde{I}_{2} \subseteq \tilde{I}_{1}$ such that $i_{2} \in  \mathcal{I}_{0}(t)$ for all $t \in \tilde{I}_{2}$.
	Applying the above procedure for all $i \in \mathcal{I}_{0}(\tau) \setminus \{ i_{1}, i_{2} \}$, one gets that there exists an open interval $\tilde{I}_{N} \subseteq \ldots \subseteq \tilde{I}_{1}$ such that $\mathcal{I}_{+}(t) = \{1, \ldots, N\}$ for all $t \in \tilde{I}_{N}$, which is a contradiction.
	Thus, $\mathcal{I}_{+}(\tau)$ is empty unless $\mathcal{I}_{0}(\tau) \cup \mathcal{I}_{-}(\tau)$ is empty.
	In fact, since $I$ and $\tau$ were chosen arbitrarily, we showed that, for any $t \ge t_{0}$, either $\mathcal{I}_{0}(t) \cup \mathcal{I}_{-}(t)$ or $\mathcal{I}_{+}(t)$ must be empty.
	
	Let $t_{1}, t_{2} \ge t_{0}$ be such that $[t_{1}, t_{2}]$ is the maximal interval containing $\tau$ on which $\mathcal{I}_{+}(t)$ is empty.
	Clearly, $t_{1} < t_{2}$.
	Adding $N$ equations in \eqref{wob:eq:wobbler}, due to pairwise symmetry of the ambient conservative force, we get that $\dot{\vw}^{\star}(t) = 0$ on $[t_{1}, t_{2}]$. 
	Hence, 
	\begin{equation}
		\label{wob:eq:constant-vel}
		\vx_{i}^{\star}(t) = \vx_{i}^{\star}(t_{1}) + \v{C}(t - t_{1}), 
		\;\;
		i = 1, \ldots, N,
	\end{equation}
	for all $t \in [t_{1}, t_{2}]$ and some constant $\v{C} \in \R^{d}$, $\v{C} \ne \vzero$.
	Since $\vx_{i}^{\star}(t)$ defined by \eqref{wob:eq:constant-vel} is unbounded as $t \to \infty$, $t_{2}$ must be finite.
	Therefore, there exists $t_{3} > t_{2}$ such that $\mathcal{I}_{+}(t_{3})$ is non-empty.
	Thus, (i) has to be true.
	
	Now, suppose that (ii) is false and let $I \subseteq (t_{0}, \infty)$ be open and such that (i) holds on it, but $\vd(t) = \vd_{0}$, for all $t \in I$ and some constant $\vd_{0} \in \R^{d}$.
	It follows from \eqref{wob:eq:u-shift} that 
	\begin{equation*}
		\vu\left(\vx_{i}^{\star} (t_{1})\right) = \vu\left(\vx_{i}^{\star} (t_{2})\right)
	\end{equation*}
	for $i = 1, \ldots, N$ and all $t_{1}, t_{2} \in I$.
	Since $\vu$ is one-to-one, $\vx_{i}^{\star}(t)$ is constant on $I$, and hence $\vw^{\star}(t) = 0$ on $I$.
	This contradicts the definition of a wobbler. 
	Thus, (ii) must be true.
\end{proof}

\begin{lemma}
	\label{wob:lemma:two-sols}
	Suppose that $h(s)$ is twice differentiable and that $h(s)$, $sh^{\prime}(s)$, and 
	$
	\frac{s \left(h^{\prime} (s)\right)^2}{2h^{\prime}(s) + \left( s h^{\prime} (s)\right)^{\prime}}	
	$
	are strictly increasing on $(r_{0}, \infty)$.
	Then, for any $a, b \in \R$, the equation 
	\begin{equation}
		\label{wob:sheq}
		\left[h(s) - a\right] \left[h(s) - a + s h^\prime(s)\right] = -b^2
	\end{equation}
	has at most two solutions.
\end{lemma}
\begin{proof}
	Define 
	\begin{equation*}
		F(s) = \left[h(s) - a\right] \left[h(s) - a + s h^\prime(s)\right],
		\;\;
		s > r_{0},
	\end{equation*}
	and let $s_{1}$, $s_{2} > r_{0}$ be such that $h(s_{1}) = a$ and $h(s_{2}) + s_{2} h^\prime(s_{2}) = a$, so that $F(s_{1}) = F(s_{2}) = 0$.
	Since both $h(s)$ and $s h^\prime(s)$ are increasing, we have that $s_{1} < s_{2}$.
	Additionally, we have that $F(s) < 0$ on the interval $(s_{1}, s_{2})$, and $F(s) > 0$ otherwise.
	Thus, if a solution of \eqref{wob:sheq} exists, it has to belong to the interval $[s_{1}, s_{2}]$.
	
	If $b = 0$, then $s_{1}$ and $s_{2}$ are the solutions of \eqref{wob:sheq}, and no other solutions exist since $h(s)$ and $s h^\prime(s)$ are increasing.
	Thus, suppose that $b \ne 0$. 
	By Rolle's theorem, $F$ has a critical point $s_{c}$ on $(s_{1}, s_{2})$. 
	We have
	\begin{equation*}
		0 = F^{\prime} (s_{c}) 
		= h^{\prime}(s_{c}) \left[h(s_{c}) - a + s_{c} h^\prime(s_{c})\right] 
		+ \left[h(s_{c}) - a\right] \left[h^{\prime}(s_{c}) + \left( s_{c} h^{\prime} (s_{c}) \right)^{\prime}\right],
	\end{equation*}
	which yields 
	\begin{equation}
		\label{wob:eq:critical}
		h(s_{c}) - a
		= - \frac{s \left(h^{\prime} (s_{c})\right)^2}{2h^{\prime}(s_{c}) + \left( s_{c} h^{\prime} (s_{c})\right)^{\prime}}.
	\end{equation}
	By assumptions of the lemma, the left-hand side is increasing and the right-hand side is decreasing, and so $s_{c}$ is the unique solution of \eqref{wob:eq:critical}.
	Thus, $F$ has only one critical point on $(s_{1}, s_{2})$ and $F(s) = -b^{2}$ has at most two solutions on $(s_{1}, s_{2})$. 
	This completes the proof.
\end{proof}
We now prove the main result of this section.
\begin{theorem}
	\label{wob:thm:non-existence}
	Suppose that $h(s)$ satisfies the assumptions of Lemma \ref{wob:lemma:two-sols} and that $N > 2$. 
	Then wobblers do not exist.
\end{theorem}
\begin{proof}
	Suppose that $(\vX^{\star}, \vW^{\star})$ is a wobbler.
	By Proposition \ref{wob:prop:wobblers}, there exist $t_{1}, t_{2} \in \R$, $t_{0} \le t_{1} < t_{2}$, such that $\vd(t_{1}) \ne \vd(t_{2})$ and that $\abs{\vx_{i}^{\star}(t)} \in B_{i} \subseteq \R^{d}$, where $B_{i}$ is an open ball lying outside of the ball $B(0, r_{0})$, for all $t \in [t_{1}, t_{2}]$ and $i = 1, \ldots, N$.
	
	Fix $i \in \{1, \ldots, N\}$ and let $\vy_{1} = \vx_{i}^{\star}(t_{1})$, $\vy_{2} = \vx_{i}^{\star}(t_{2})$, $\vy^{\star} = \vy_{2} - \vy_{1}$, $\vd^{\star} = \vd(t_{2}) - \vd(t_{1})$.
	Using a first-degree Taylor expansion of $\vu$ at $\vy_{1}$, we can write \eqref{wob:eq:rigidity} as 
	\begin{equation}
		\label{wob:eq:linearized}
		\vd^{\star}
		= \vu\left(\vy_{2}\right) - \vu\left(\vy_{1}\right)
		= \nabla \vu (\vy) \vy^{\star} 
		= - \abs{\vy} h^{\prime} (\abs{\vy}) \left( \frac{\vy}{\abs{\vy}} \otimes \frac{\vy}{\abs{\vy}} \right) \vy^{\star}
		- h (\abs{\vy}) \vy^{\star},
	\end{equation}
	where $\vy$ belongs to the line segment connecting $\vy_{1}$ and $\vy_{2}$, so that $\vy \in B_{i}$ and $\vy \ne \vzero$.
	Write \eqref{wob:eq:linearized} as 
	\begin{equation}
		\label{wob:eq:linearized-dot}
		\abs{\vy} h^{\prime} (\abs{\vy})
		\left(\frac{\vy}{\abs{\vy}}  \cdot \vy^{\star} \right) \frac{\vy}{\abs{\vy}}
		= \vp,
	\end{equation}
	where $\vp = - \vd^{\star} - h(\abs{\vy})\vy^{\star}$. 
	Since $\frac{\vy}{\abs{\vy}}$ is a unit vector that is collinear with $\vp$, we can write
	\begin{equation}
		\label{wob:eq:linearized-p}
		\abs{\vy} h^{\prime} (\abs{\vy})
		\left(\frac{\vp}{\abs{\vp}} \cdot \vy^{\star} \right) \frac{\vp}{\abs{\vp}}
		= \vp.
	\end{equation}
	Note that $\vy^{\star} \ne \vzero$ since $\vd^{\star} \ne \vzero$ and $\vu$ is one-to-one.
	Also, $h^{\prime}(s) > 0$ whenever $s > r_{0}$ by the assumption of the theorem.
	Thus, \eqref{wob:eq:linearized-p} is equivalent to
	\begin{equation}
		\label{wob:eq:linearized-1d}
		\abs{\vy} h^{\prime} (\abs{\vy})
		\left(\vp \cdot \vy^{\star} \right) 
		= \abs{\vp}^2
	\end{equation}
	for all $\vp \ne \vzero$. 
	
	Suppose that $\vp = \vzero$, so that $\vd^{\star} = - h(\abs{\vy})\vy^{\star}$.
	Then \eqref{wob:eq:rigidity} yields
	\begin{equation*}
		- h(\abs{\vy})\vy^{\star} 
		= - h(\abs{\vy_{1} + \vy^{\star}})(\vy_{1} + \vy^{\star}) + h(\abs{\vy_{1}})\vy_{1}.
	\end{equation*}
	Rearranging terms, we get
	\begin{equation}
		\label{wob:eq:linearized-linindep}
		\left[h(\abs{\vy_{1} + \vy^{\star}}) - h(\abs{\vy_{1}})\right] \vy_{1} 
		+ \left[h(\abs{\vy_{1} + \vy^{\star}}) - h(\abs{\vy})\right] \vy^{\star}
		= \vzero.
	\end{equation}
	If $\vy_{1}$ is collinear with $\vy^{\star}$, then so are $\vy_{2} = \vy_{1} + \vy^{\star}$ and $\vy$.
	However, it follows from \eqref{wob:eq:linearized-dot} that $\vy \cdot \vy^{\star} = \vzero$, and thus $\vy_{1}$ and $\vy^{\star}$ have to be linearly independent.
	Then \eqref{wob:eq:linearized-linindep} implies that
	\begin{equation*}
		h(\abs{\vy_{1} + \vy^{\star}}) - h(\abs{\vy_{1}}) 
		= h(\abs{\vy_{1} + \vy^{\star}}) - h(\abs{\vy})
		= 0,
	\end{equation*}
	and hence
	\begin{equation*}
		\abs{\vy_{1} + \vy^{\star}} = \abs{\vy_{1}} = \abs{\vy},
	\end{equation*}
	which is not possible since $\vy$ belongs to the line segment connecting $\vy_{1}$ and $\vy_{1} + \vy^{\star}$.
	Therefore $\vp \ne 0$.
	
	Let $\vy^{\perp}$ be such that $\vy^{\star} \cdot \vy^{\perp} = 0$ and $\vd^{\star} = -a \vy^{\star} - \vy^{\perp}$ for some $a \in \R$.
	Then \eqref{wob:eq:linearized-1d} can be written as
	\begin{equation*}
		s h^{\prime}(s)
		\left(a \vy^{\star} + \vy^{\perp} - h(s) \vy^{\star}\right) \cdot \vy^{\star}
		= \abs{a \vy^{\star} + \vy^{\perp} - h(s) \vy^{\star} }^2,
	\end{equation*}
	from which we get
	\begin{equation*}
		\abs{\vy} h^{\prime}(\abs{\vy}) \left[h(\abs{\vy}) - a\right] \abs{\vy^{\star}}^{2} 
		= \left[h(\abs{\vy}) - a\right]^{2} \abs{\vy^{\star}}^{2} + \abs{\vy^{\perp}}^2,
	\end{equation*}
	and so
	\begin{equation}
		\label{wob:eq:two-sols}
		\left[h(\abs{\vy}) - a\right]^{2} + \abs{\vy} h^{\prime}(\abs{\vy}) \left[h(\abs{\vy}) - a\right] 
		= - \frac{\abs{\vy^{\perp}}^2}{ \abs{\vy^{\star}}^{2} }.
	\end{equation}
	By Lemma \ref{wob:lemma:two-sols}, \eqref{wob:eq:two-sols} has at most two distinct solutions for $\abs{\vy}$.
	Furthermore, since $\vy$ and $\vp$ are collinear and $\vp$ is completely determined by $\abs{\vy}$, so is $\vy$.
	Thus, \eqref{wob:eq:linearized-p} has at most two solutions for $\vy$.
	Finally, since \eqref{wob:eq:linearized-p} depends on $i$ only through $\vy$, it follows that $(\vX^{\star}, \vW^{\star})$ does not exist, provided $N > 2$.
\end{proof}
As it can be seen from the following example, the assumptions of Lemma \ref{wob:lemma:two-sols} are easy to satisfy. 
\begin{example}
	\normalfont
	Suppose that $k_{1}(s) = s^{\gamma + 1}$, $\gamma > 0$, $s \ge r_{1}$ for some $r_{1} > r_{0}$. 
	Then
	\begin{equation*}
		h_{1}(s) = \frac{k_{1}(s)}{s} = s^{\gamma},
		\;\;
		sh_{1}^{\prime}(s) = \gamma s ^{\gamma},
		\;\;
		\frac{s \left(h_{1}^{\prime} (s)\right)^2}{2h_{1}^{\prime}(s) + \left(s h_{1}^{\prime} (s)\right)^{\prime}}
		= \frac{\gamma}{\gamma + 2} s^{\gamma},
	\end{equation*}
	are all increasing on $(r_{1}, \infty)$. 
	If one puts $h_{2}(s) = s e^{s}$, $s \ge r_{2}$ for some $r_{2} > r_{0}$, then
	\begin{equation*}
		h_{2}(s) = \frac{k_{2}(s)}{s} = e^{s},
		\;\;
		sh_{2}^{\prime}(s) = s e^{s},
		\;\;
		\frac{s \left(h_{2}^{\prime} (s)\right)^2}{2h_{2}^{\prime}(s) + \left(s h_{2}^{\prime} (s)\right)^{\prime}}
		= \frac{s}{s + 3} e^{s}.
	\end{equation*}
	One can verify that the above functions are increasing on $(r_{2}, \infty)$ by computing the corresponding derivatives.
	Similarly, if $k_{3}(s) = s \ln(s + 1)$, $s \ge r_{3}$ for some $r_{3} > r_{0}$, then
	\begin{equation*}
		h_{3}(s) = \frac{k_{3}(s)}{s} = \ln(s + 1),
		\;\;
		sh_{3}^{\prime}(s) = \frac{s}{s + 1},
		\;\;
		\frac{s \left(h_{3}^{\prime} (s)\right)^2}{2h_{3}^{\prime}(s) + \left(s h_{3}^{\prime} (s)\right)^{\prime}}
		= \frac{s}{2s + 3},
	\end{equation*}
	are increasing on $(r_{3}, \infty)$.
	Thus, it is easy to guarantee the non-existence of wobblers for a large class of nonlinear position alignment forces.
\end{example}

\subsection{Linear case: existence}

Now, suppose that $r_{0} = 0$ and $k(s) = s$, so that
\begin{equation*}
	\vu(\vx) = -\alpha \vx
\end{equation*}
is linear. 
Equations \eqref{wob:eq:wobbler} can be written as
\begin{equation}
	\label{wob:eq:lin-wob}
	\ddot{\vx}^{\star} + \alpha \vx^{\star}
	= \vf_{i}^{C} - \alpha \vx_{i}^{0},
	\;\;
	i = 1, \ldots, N.
\end{equation}
Since the left-hand side does not depend on the index $i$, and the right-hand side does not depend on $t$, the right-hand side is a constant vector, which we denote by $\vf^{0}$. 
We assume that the equations of motion are supplemented with the initial conditions
\begin{equation}
	\label{wob:eq:lin-inits}
	\vx^{\star}(t_0)=0,\qquad \dot{\vx^{\star}}(t_0)=\vw^{\star} (0).
\end{equation}
Solving the characteristic equation for \eqref{wob:eq:lin-wob} yields the fundamental set of solutions $\cos\mu t$, $\sin\mu t$, where $\mu = \alpha^{1/2}$.
The solution of the IVP \eqref{wob:eq:lin-wob}, \eqref{wob:eq:lin-inits} is
\begin{equation}
	\label{wob:eq:linear-wob-sol}
	\vx^{\star}(t)=\v{a}\cos(\mu t)+\v{b}\sin(\mu t) + \alpha^{-1} \vf^{0},
\end{equation}
where 
\begin{eqnarray}
	\label{wob:eq:linear-wob-sol-a}
	\v{a}& = & - \alpha^{-1} \cos \mu t_0 \vf^{0} - \mu^{-1} \sin \mu t_0 \vw_0^{\star},\\
	\label{wob:eq:linear-wob-sol-b}
	\v{b}& = & \mu^{-1} \cos \mu t_0 \vw_0^{\star}-m^{-1} \sin \mu t_0 \vf^{0}.
\end{eqnarray}
Equations \eqref{wob:eq:linear-wob-sol}-\eqref{wob:eq:linear-wob-sol-b} describe a periodic translation of the initial array $\vX_0^{\star}=(\vx_1^{0},\ldots, \vx_N^{0})$. 
The dependence of this array on the parameter $\vf_{0}$, is prescribed by solving the quasi-equilibrium equations
\begin{equation*}
	\vf_{i}^C(\vX_0^{\star}) + \vu(\vx_{i}^{0}) = \vf^{0}.
\end{equation*}
for any $\vf^{0}$ that would yield $\vX^{\star}, \vW^{\star}$ compatible with \eqref{wob:eq:w-bounded} and \eqref{wob:eq:x-bounded}.

\begin{example}
	\normalfont
	Suppose that $t_{0} = 0$, and consider an equilibrium solution $(\vf^{0} = \vzero)$ as the initial array.
Then, for each admissible $\vw^{\star}(0) \ne \vzero$, the initial array undergoes a rigid periodic translation 
\begin{equation}
	\label{wob:eq:linear-periodic}
	\vx_{i}^{\star}(t) = \vx_{i}^{0} + \mu^{-1} \sin(\mu t) \vw^{\star}(0), 
	\;\; 
	i=1, \ldots, N.
\end{equation}
The inequalities \eqref{wob:eq:w-bounded} and \eqref{wob:eq:x-bounded} impose restrictions on $\vx_{i}^{0}$, $\mu$, and $\vw^{\star}(0)$. 
In particular, 
\begin{equation*}
	\abs{\vx_{i}^{0}} \le R
\end{equation*}
and
\begin{equation*}
	\abs{\vx_{i}^{0} + \mu^{-1} \vw^{\star}(0)} \le R,
\end{equation*}
where $\vw^{\star}(0)$ must also satisfy \eqref{wob:eq:w-bounded}.
\end{example}

%% file: sections/05_non-dissipative.tex
\section{Non-dissipative dynamics}
\label{section:non-diss}

In this section, we investigate flocking in the case when the total mechanical energy cannot be guaranteed to be non-increasing in time, and the corresponding dynamics may be non-dissipative. In such cases, the Lyapunov-like argument used earlier is no longer applicable.

Since flocking cannot be expected to occur in the most general case, suitable assumptions must be imposed on the forces. In this section, non-linear self-propulsion forces are treated as bounded perturbations of the linear forces. This assumption is quite general, because it allows us to work with forces prescribed by arbitrary bounded functions on a bounded region in the phase space. As the phase space distance from the origin increases, these forces should be approximately linear in the sense of Definition 3 given below.
Another important example of forces that can produce non-dissipative dynamics is the force $M \boldsymbol{f}^L$ associated with the acceleration of the virtual leader. 

From the practical standpoint, our results are applicable to (i) leader trajectories with arbitrary bounded accelerations, such as circular motions with bounded angular velocities, spiraling, etc.; (ii) non-dissipative velocity alignment forces whose power (rate of work) may change sign over time; and (iii) random forces given by bounded random variables. For such forces, we provide bounds for every realization, considered "one realization at a time." The size of the absorbing ball remains fixed as long as the bounds on the force realizations are fixed.

The central analytical result is Theorem 3, which provides time-independent bounds for the positions and velocities of agents. This implies the existence of an absorbing ball in the phase space, which every trajectory enters at a certain time and thereafter remains within the ball. Once the existence of the absorbing ball is established, standard arguments (see, e.g., \cite{temam1997infinite}) can be used to prove the existence of a global attractor. This in turn implies that the system exhibits approximate flocking. 

Next, we show that, for some special configurations of the tunable control parameters, the flocking becomes exact and proper.
We then consider the structural dynamics of the group and show that the velocities of the agents converge to the velocity of the center of mass at an exponential rate under some mild additional assumptions on the nonlinearities of the self-propulsion forces, making the flocking exact for a broader range of control parameters.
Finally, we show that the deviation of the trajectory of the center of mass from that of the virtual leader is bounded.

\subsection{Boundedness of solutions}

The first assumption we make is that the acceleration of the virtual leader is bounded, which is natural to expect in any practical application.
\begin{assumption}
	\label{non-diss:assmp:fl-is-bdd}
	There exists $C_{l} \ge 0$ such that $\abs{\vf^{L}(t)} \le C_{l}$, for all $t \ge 0$.
\end{assumption}
In addition, we restrict our attention to the self-propulsion forces that are bounded perturbations of the linear ones.
\begin{definition}
	\label{non-diss:def:approx-linear}
	A function $\vf: \R^{d} \to \mathbb{R}^{d}$ is \textbf{approximately linear} if there exist real constants $K \ne 0$ and $C>0$ such that
	\begin{equation*}
		\abs{\vf(\vy) - K \vy } \leq C
	\end{equation*}
	for all $\vy \in \R^{d}$ and $t \in [0, \infty)$.
\end{definition}
\begin{assumption}
	\label{non-diss:assmp:sp-forces-are-approx-linear}
	The position alignment force $\vu^{P}$ and the velocity alignment force $\vu^{V}$ are approximately linear with constants $-K_{p}$, $C_{p}$ and $-K_{v}$, $C_{v}$, respectively, where $K_{p}, K_{v} > 0$.
\end{assumption}
Note that Definition \ref{non-diss:def:approx-linear} requires a function to be only ``asymptotically linear'', allowing for arbitrary behavior for bounded values of the argument.
Consequently, the class of approximately linear functions contains a broad range of nonlinear control protocols that may be of practical interest.

Let
\begin{equation*}
	\delta \vu^{P}(\vy) = \vu^{P}(\vy) + K_{p} \vy, 
	\;\;
	\vy \in \R^{d},
\end{equation*}
and
\begin{equation*}
	\delta \vu^{V}(\vy) = \vu^{V}(\vy) + K_{v} \vy, 
	\;\;
	\vy \in \R^{d},
\end{equation*}
denote the nonlinear parts of the position alignment and the velocity alignment forces, respectively.
Write \eqref{diss:dynamics} as
\begin{equation}
	\label{non-diss:eq:dynamics-full}
	\begin{aligned}
		\dot{\vX} & = \vW,
		\\
		M \dot{\vW} & = - \nabla_{\vX} U(\vX) - L(\vX) \vW - K_{p} \vX - K_{v} \vW + \vF(t),
	\end{aligned}
\end{equation}
where
\begin{equation*}
	\vF(t) = \delta \vu^{P}(\vX) + \delta \vu^{V}(\vW) - M \vf^{L} \otimes \vones_{N}
\end{equation*}
and 
\begin{equation*}
	\delta \vu^{P}(\vX) = \diag \left( \delta \vu^{P}(\vx_{1}), \ldots,  \delta \vu^{V}(\vx_{N}) \right),
	\quad
	\delta \vu^{V}(\vW) = \diag \left( \delta \vu^{V}(\vv_{1}), \ldots,  \delta \vu^{V}(\vv_{N}) \right).
\end{equation*} 
Note that it follows from Assumptions \ref{non-diss:assmp:fl-is-bdd} and \ref{non-diss:assmp:sp-forces-are-approx-linear} that 
\begin{equation}
	\label{non-diss:eq:f-bound}
	\abs{\vF(t)} \le C_{1}
\end{equation}
for all $t \ge 0$, where
\begin{equation*}
	C_{1} = N^{\frac{1}{2}}(C_{p} + C_{v} + M C_{l}).
\end{equation*}

Consider the following reduced system
\begin{equation}
	\label{non-diss:eq:dynamics-reduced}
	\begin{aligned}
		\dot{\vX} & = \vW,
		\\
		M\dot{\vW} & = - \nabla_{\vX} U(\vX) -L(\vX) \vW - K_{p} \vX - K_{v}  \vW.
	\end{aligned}
\end{equation}
We first prove that the velocities of the reduced system \eqref{non-diss:eq:dynamics-reduced} decay exponentially.
\begin{lemma}
	\label{non-diss:lemma:w-reduced-decay}
	Let $(\vX(t), \vW(t))$ be a solution of the reduced system \eqref{non-diss:eq:dynamics-reduced}.
	Then, for any $t_{0} \ge 0$, 
	\begin{equation*}
		\abs{\vW(t)} \le \abs{\vW(t_{0})} e^{-\frac{K_{v}}{M}(t-t_{0})}
	\end{equation*}
	for all $t \ge t_{0}$.
\end{lemma}
\begin{proof}
	Let 
	\begin{equation}
		\label{non-diss:eq:energy-reduced-def}
		E_{r}(\vX, \vW) = \frac{1}{2} M \abs{\vW}^{2} + \vU(\vX) + \frac{1}{2} K_{p}\abs{\vX}^2.
	\end{equation}	
	Multiplying both sides of the second equation in \eqref{non-diss:eq:dynamics-reduced} by $\vW$ yields
	\begin{equation}
		\label{non-diss:eq:energy-estimate}
		\frac{d}{dt} E_{r}(\vX(t), \vW(t)) =  - \vW^{T} L(\vX) \vW - K_{v} \abs{\vW}^2 \le - K_{v} \abs{\vW}^2
	\end{equation}
	since $L(\vX)$ is positive semi-definite.
	
	Let $t_{0} \ge 0$. 
	Integrating both sides of \eqref{non-diss:eq:energy-estimate}, we get
	\begin{equation*}
		E_{r}(\vX(t), \vW(t)) \le - K_{v} \int_{t_{0}}^{t} \abs{\vW(\tau)}^{2} d\tau + E_{r}(t_{0}),
	\end{equation*}
	where $E_{r}(t_{0}) = E(\vX(t_{0}), \vW(t_{0}))$.
	Hence,
	\begin{equation}
		\label{non-diss:eq:energy-estimate-2}
		\abs{\vW(t)}^{2} \le - \frac{2 K_{v}}{M} \int_{t_{0}}^{t} \abs{\vW(\tau)}^{2} d\tau + \frac{2 E_{r}(t_{0})}{M}.
	\end{equation}
	Let $\Lambda(t) = \int_{t_{0}}^{t} \abs{\vW(\tau)}^{2} d \tau$. 
	Multiplying both sides of \eqref{non-diss:eq:energy-estimate-2} by $e^{\frac{2 K_{v}}{M}(t-t_{0})}$ yields 
	\begin{equation}
		\label{non-diss:eq:energy-estimate-3}
		\frac{d}{dt} \left( e^{\frac{2 K_{v}}{M}(t-t_{0})} \Lambda(t) \right)  \le \frac{2 E_{r}(t_{0})}{M}  e^{\frac{2 K_{v}}{M}(t-t_{0})}.
	\end{equation}
	Integrating \eqref{non-diss:eq:energy-estimate-3}, we get
	\begin{equation*}
		e^{\frac{2 K_{v}}{M}(t-t_{0})} \Lambda(t) \le  \Lambda(t_{0}) +  \frac{E_{r}(t_{0})}{K_{v}}\left( e^{\frac{2 K_{v}}{M}(t-t_{0})} - 1 \right),
	\end{equation*}
	and therefore
	\begin{equation}
		\label{non-diss:eq:energy-estimate-4}
		\int_{t_{0}}^{t} \abs{\vW(\tau)}^{2} d \tau \le \frac{E_{r}(t_{0})}{K_{v}} \left(1 -  e^{-\frac{2 K_{v}}{M}(t-t_{0})} \right)
	\end{equation}
	since $\Lambda(t_{0}) = 0$.
	Differentiating \eqref{non-diss:eq:energy-estimate-4} yields
	\begin{equation*}
		\begin{aligned}
		\abs{\vW(t)}^{2} 
		& \le \frac{2 E_{r}(t_{0})}{M} e^{-\frac{2 K_{v}}{M}(t-t_{0})}
		\\
		& \le E_{r}(t_{0}) \left(\frac{1}{2}M\abs{\vW(t_{0})}^{2}\right)^{-1} \abs{\vW(t_{0})}^{2} e^{-\frac{2 K_{v}}{M}(t-t_{0})}
		\\
		& \le \abs{\vW(t_{0})}^{2} e^{-\frac{2 K_{v}}{M}(t-t_{0})},
		\end{aligned}
	\end{equation*}
	and the result follows.
\end{proof}

Now, observe that analogously to the ambient dissipative term $-L(\vW)\vW$, the ambient conservative term $- \nabla_{\vX} U(\vX)$ in \eqref{non-diss:eq:dynamics-full} and \eqref{non-diss:eq:dynamics-reduced} can be written as $L_{C}(\vX)\vX$, where $L_{C} (\vX) = \mathcal{L}_{C}(\vX) \otimes I_{d}$.
Here, $\mathcal{L}_{C}(\vX)$ is the Laplacian matrix of the position-dependent undirected weighted graph $\mathcal{G}_{C}(\v{X}) = \left( \mathcal{A}, \mathcal{E}_{C}(\vX), \sigma_{C} \right)$, with $\mathcal{A} = \{1 , \ldots , N\}$, $\mathcal{E}_{C}(\vX) = \{ (i, j) \in \mathcal{A} \times \mathcal{A} : \abs{\vx_{ij}} \le r_{C} \}$, and $\sigma_{C}(i, j) = A w_{C} (\abs{\vx_{ij}})$.
This means that \eqref{non-diss:eq:dynamics-reduced} can be viewed as a linear non-autonomous system, while \eqref{non-diss:eq:dynamics-full} can be viewed as its non-homogeneous counterpart.
This observation along with Lemma \ref{non-diss:lemma:w-reduced-decay} allows us to show that the velocities of the full system \eqref{non-diss:eq:dynamics-full} are bounded.

\begin{lemma}
	\label{non-diss:lemma:w-full-bdd}
	Let $(\vX(t), \vW(t))$ be a solution of the full system \eqref{non-diss:eq:dynamics-full}.
	Suppose that Assumptions \ref{non-diss:assmp:fl-is-bdd} and \ref{non-diss:assmp:sp-forces-are-approx-linear} hold.
	Then, for any $\varepsilon_{0} > 0$, there exists $t_{0} = t_{0}(\vW(0), \varepsilon_{0}) \ge 0$ such that 
	\begin{equation*}
		\abs{\vW(t)} \le \varepsilon_{0} + \frac{C_{1}}{K_{v}}
	\end{equation*} 
	for all $t \ge t_{0}$.
\end{lemma}
\begin{proof}
	Let $\Psi(t, s)$, $t \ge s \ge 0$, be the state-transition matrix of the linear system \eqref{non-diss:eq:dynamics-reduced}.
	By the variation of constants formula, the solution of \eqref{non-diss:eq:dynamics-full} is given by
	\begin{equation}
		\label{non-diss:eq:duhamel}
		(\vX(t), \vW(t)) =
		\Psi(t, 0) (\vX(0), \vW(0))
		+ \int_{0}^{t} \Psi(t, \tau) (\v0, M^{-1}\vF(\tau)) d\tau,
	\end{equation}
	Let $\Psi_{\vW}(t, s)$ be the last $Nd$ rows of be the matrix $\Psi(t, s)$, i.e., $\Psi_{\vW}(t, s)$ is such that the velocity components of the solution of \eqref{non-diss:eq:dynamics-reduced} are given by $\Psi_{\vW}(t, s) (\vX(s), \vW(s))$.
	Then, it follows from \eqref{non-diss:eq:f-bound}, \eqref{non-diss:eq:duhamel}, and Lemma \eqref{non-diss:lemma:w-reduced-decay} that
	\begin{equation*}
		\begin{aligned}
			\abs{\vW(t)} & \le \abs{\Psi_{\vW}(t, 0) (\vX(0), \vW(0))} + \abs{\int_{0}^{t} \Psi_{\vW}(t, \tau)_{\vW} (\v0, M^{-1} \vF(\tau)) d\tau }
			\\
			& \le \abs{\vW(0)} e^{-\frac{K_{v}}{M} t} + \frac{1}{M} \int_{0}^{t} \abs{\vF(\tau)} e^{-\frac{K_{v}}{M} \tau} d\tau
			\\
			& = \left(\abs{\vW(0)} - \frac{C_{1}}{K_{v}} \right) e^{-\frac{K_{v}}{M} t} + \frac{C_{1}}{K_{v}}.
		\end{aligned}
	\end{equation*}
	Since $K_{v} > 0$, the result follows.
\end{proof}

Equipped with the above lemma we are ready to prove that the solutions of the full system \eqref{non-diss:eq:dynamics-full} are bounded.
\begin{theorem}
	\label{non-diss:thm:bdd-sols}
	Let $(\vX(t), \vW(t))$ be a solution of \eqref{non-diss:eq:dynamics-full}.
	Suppose that Assumptions \ref{non-diss:assmp:fl-is-bdd} and \ref{non-diss:assmp:sp-forces-are-approx-linear} hold, and that $\tilde{K}_{v}^{2} - 4\tilde{K}_{p} \ne 0$, where $\tilde{K}_{p} = M^{-1} K_{p}$, $\tilde{K}_{v} = M^{-1} K_{p}$.
	Then, for $i = 1, \ldots, N$ and any $\varepsilon_{1} > 0$, there exists $t_{1} = t_{1}(\vx_{i}(0), \vw_{i}(0), \varepsilon_{1}) \ge 0$ such that
	\begin{enumerate}[(i)]
		\item $\abs{\vx_{i}(t)} \le \varepsilon_{1} + \frac{2 \tilde{K_{v}}}{\tilde{K_{p}} \sqrt{\abs{\tilde{K_{v}}^{2} - 4\tilde{K}_{p}}}} C_{2}$;
		\item $\abs{\vw_{i}(t)} \le \varepsilon_{1} + \frac{2}{\sqrt{\abs{\tilde{K}_{v}^{2} - 4\tilde{K}_{p}}}} C_{2}$;
	\end{enumerate}
	for all $t \ge t_{1}$, where
	\begin{equation*}
		C_{2} = \frac{1}{M} \left[ (N-1)^{\frac{1}{2}} A r_{C} + (N^{2}-N)^{\frac{1}{2}} \frac{B C_{1}}{K_{v}}  + N^{-\frac{1}{2}}C_{1} \right].
	\end{equation*}
\end{theorem}
\begin{proof}
	Write \eqref{non-diss:eq:dynamics-full} as
	\begin{equation}
		\label{non-diss:eq:bdd-sols-dynamics}
		\begin{aligned}
			\dot{\vx}_{i} & = \vw_{i},
			\\
			M \dot{\vw}_{i} & = - K_{p} \vx_{i} - K_{v} \vw_{i} + \vH_{i}(t),
		\end{aligned}
		\quad
		i = 1, \ldots, N,
	\end{equation}
	where 
	\begin{equation*}
		\v{H}_{i}(t) = \vf^{C}_{i} + \vf^{D}_{i} + \vF_{i}(t).
	\end{equation*}
	
	Let $\varepsilon_{0} > 0$.
	It follows from Lemma \ref{non-diss:lemma:w-full-bdd} that there exists $t_{0} = t_{0}(\vW(0), \varepsilon_{0}) \ge 0$ such that 
	\begin{equation*}
		\abs{\v{H}_{i}(t)}
		=
		\abs{
			A \sum_{j \ne i} w_{C} \left(\abs{\vx_{ij}}\right) \vx_{ij} 
			- B \sum_{j \ne i} w_{D} (\abs{\vx_{ij}}) \vw_{ij}
			+ \vF_{i}(t)
		}
		\le C_{H} (\varepsilon_{0}),
	\end{equation*}
	for $i = 1, \ldots, N$ and all $t \ge t_{0}$, where
	\begin{equation}
		\label{non-diss:eq:c-h}
		C_{H} (\varepsilon_{0})
		=
		(N-1)^{\frac{1}{2}} A r_{C} + (N^{2}-N)^{\frac{1}{2}} B \left(\varepsilon_{0} + \frac{C_{1}}{K_{v}}\right)  + N^{-\frac{1}{2}}C_{1}.
	\end{equation}
	
	Viewing \eqref{non-diss:eq:bdd-sols-dynamics} as a linear non-homogeneous system, for given $i \in \{1, \ldots, N\}$, we have
	\begin{eqnarray}
		\label{non-diss:eq:bdd-sols-duhamel-pos}
		\vx_{i}(t) & = & \frac{e^{\lambda_{1}(t - t_{0})}}{\lambda_{2} - \lambda_{1}} \left[\lambda_{2} \vx_{i}(t_{0}) - \vw_{i}(t_{0})\right] +
		\frac{e^{\lambda_{2} (t - t_{0})}}{\lambda_{2} - \lambda_{1}} \left[ -\lambda_{1}  \vx_{i}(t_{0}) +  \vw_{i}(t_{0}) \right] 
		\nonumber \\
		& + &  \frac{1}{\lambda_{2} - \lambda_{1}} \int_{t_{0}}^{t} \left[ - e^{\lambda_{1}(t-\tau)} + e^{\lambda_{2}(t-\tau)} \right] M^{-1} \vH_{i}(\tau)d\tau, 
		\\
		\label{non-diss:eq:bdd-sols-duhamel-vel}
		\vw_{i}(t) & = &  \frac{e^{\lambda_{1}(t - t_{0})}}{\lambda_{2} - \lambda_{1}} \left[\lambda_{1}\lambda_{2} \vx_{i}(t_{0}) - \lambda_{1} \vw_{i}(t_{0})\right] +
		\frac{e^{\lambda_{2}(t - t_{0})}}{\lambda_{2} - \lambda_{1}} \left[-\lambda_{1}\lambda_{2} \vx_{i}(t_{0}) + \lambda_{2}\vw_{i}(t_{0}) \right] 
		\nonumber \\
		& + &  \frac{1}{\lambda_{2} - \lambda_{1}}  \int_{0}^{t} \left[ -\lambda_{1} e^{\lambda_{1}(t-\tau)} + \lambda_{2} e^{\lambda_{2}(t-\tau)} \right] M^{-1} \vH_{i}(\tau) d\tau, 
	\end{eqnarray}
	where
	\begin{equation}
		\label{non-diss:eq:bdd-sols-evalues}
		\lambda_{1, 2}=\frac{-\tilde{K_v}\pm \sqrt{ \tilde{K_v}^2-4 \tilde{K_p}}}{2}.
	\end{equation}
	
	Since the real parts of $\lambda_{1, 2}$ are negative, first two terms in \eqref{non-diss:eq:bdd-sols-duhamel-pos} and \eqref{non-diss:eq:bdd-sols-duhamel-vel} vanish as $t \to \infty$, and therefore the asymptotic behavior of $\vx_{i}(t)$ and $\vw_{i}(t)$ is determined by the corresponding integral terms.
	
	For $t \ge t_{0}$, for the integral term in \eqref{non-diss:eq:bdd-sols-duhamel-pos}, we have
	\begin{equation*}
		\abs{ 
			\frac{1}{\lambda_{2} - \lambda_{1}}  
			\int_{t_{0}}^{t} \left[ - e^{\lambda_{1}(t-\tau)} + e^{\lambda_{2}(t-\tau)} \right] M^{-1} \v{H}_{i}(\tau)d\tau 
		}
		\le 
		\abs{\frac{2(\lambda_{1} + \lambda_{2})}{(\lambda_{2} - \lambda_{1}) \lambda_{1} \lambda_{2}}} M^{-1} C_{H}(\varepsilon_{0}),
	\end{equation*}
	and for the integral term in \eqref{non-diss:eq:bdd-sols-duhamel-vel}, we have
	\begin{equation*}
		\abs{ 
			\frac{1}{\lambda_{2} - \lambda_{1}}  
			\int_{t_{0}}^{t} \left[ -\lambda_{1} e^{\lambda_{1}(t-\tau)} + \lambda_{2} e^{\lambda_{2}(t-\tau)} \right] M^{-1} \v{H}_{i}(\tau)d\tau 
		}
		\le 
		\frac{2}{\abs{\lambda_{2} - \lambda_{1}}} M^{-1} C_{H}(\varepsilon_{0}).
	\end{equation*}
	Then (i) and (ii) follow.
\end{proof}

\begin{remark}
	\normalfont
	Let $\gamma = \frac{4 \tilde{K}_{p}}{\tilde{K}_{v}^{2}}$.
	It follows from Theorem \ref{non-diss:thm:bdd-sols} that the coordinates of the $i$-th agent satisfy
	\begin{equation}
		\label{non-diss:eq:bdd-sols-remark-bound}
		\abs{\vx_{i}(t)} \le \frac{2}{\tilde{K}_{p} \sqrt{\abs{1 - \gamma}}} C_{2} \nu ,
		\qquad
		\abs{\vw_{i}(t)} \le \frac{2}{\tilde{K}_{v} \sqrt{\abs{1 - \gamma}}} C_{2} \nu,
	\end{equation}
	for all $t \ge 0$, where $\nu > 1$ depends on $\vx_{i}(0)$ and $\vw_{i}(0)$ and does not depend on time.
	By tuning the parameters $K_{p}$ and $K_{v}$, one can control the deviation of the agent's trajectory from that of the virtual leader.
	In particular, if $K_{p} \gg K_{v}$, so that $\gamma > 1$ and the eigenvalues \eqref{non-diss:eq:bdd-sols-evalues} become complex, the dynamics of the agents will be more oscillatory. 
	At the same time, higher values of $K_{p}$ will ensure a tighter bound on the agent's position.
	In turn, if $K_{p} \ll K_{v}$, we get $\gamma < 1$, which makes the eigenvalues \eqref{non-diss:eq:bdd-sols-evalues} real.
	Such a scenario will represent an ``overdamped'' system with fewer oscillations and shorter stabilization time.
	Notably, for $\gamma$ being close to 1, the bounds in \eqref{non-diss:eq:bdd-sols-remark-bound} become small, and therefore it is possible to ensure a tight formation control of the group while using relatively small values of $K_{p}$ and $K_{v}$.
	The case when $\gamma = 1$ is discussed in the theorem that follows.
\end{remark}

For the case of repeated eigenvalue in \eqref{non-diss:eq:bdd-sols-dynamics}, we can get a stronger result than that of Theorem \ref{non-diss:thm:bdd-sols}.
\begin{theorem}
	\label{non-diss:thm:bdd-sols-repeated}
	Let $(\vX(t), \vW(t))$ be a solution of \eqref{non-diss:eq:dynamics-full}.
	Suppose that Assumptions \ref{non-diss:assmp:fl-is-bdd} and \ref{non-diss:assmp:sp-forces-are-approx-linear} hold, and that $\tilde{K}_{v}^{2} - 4\tilde{K}_{p} = 0$.
	Then, for $i = 1, \ldots, N$,
	\begin{enumerate}[(i)]
		\item for any $\varepsilon_{1} > 0$, there exists $t_{1} = t_{1}(\vx_{i}(0), \vw_{i}(0), \varepsilon_{1}) \ge 0$ such that $\abs{\vx_{i}(t)} \le \varepsilon_{1} + \frac{4}{\tilde{K_{v}}}C_{2}$ for all $t \ge t_{1}$;
		\item $\abs{\vw_{i}(t)} \to 0$ as $t \to \infty$.
	\end{enumerate}
\end{theorem}
\begin{proof}
	The proof is analogous to that of Theorem \ref{non-diss:thm:bdd-sols}. 
	As before, consider the linear non-homogeneous system \eqref{non-diss:eq:bdd-sols-dynamics}. 
	When $\tilde{K}_{v}^{2} - 4\tilde{K}_{p} = 0$, the corresponding linear homogeneous system has a single repeated eigenvalue $\lambda = - \frac{\tilde{K_{v}}}{2}$.
	
	Fix $\varepsilon_{0} > 0$. For given $i \in \{1, \ldots, N\}$, we have
	\begin{eqnarray}
		\label{non-diss:eq:bdd-sols-duhamel-pos-repeated}
		\vx_{i}(t) & = & e^{\lambda (t - t_{0})}\left[ \left(-\lambda(t-t_{0}) + 1 \right) \vx_{i}(t_{0}) + (t - t_{0}) \vw_{i}(t_{0}) \right]
		\nonumber \\
		& + &  \int_{t_{0}}^{t} e^{\lambda(t - \tau)} (t - \tau) M^{-1} \vH_{i}(\tau)d\tau, 
		\\
		\label{non-diss:eq:bdd-sols-duhamel-vel-repeated}
		\vw_{i}(t) & = &  e^{\lambda (t - t_{0})}\left[ -\lambda^{2}(t-t_{0}) \vx_{i}(t_{0}) + \left(\lambda (t - t_{0}) + 1\right) \vw_{i}(t_{0}) \right]
		\nonumber \\
		& + &  \int_{t_{0}}^{t} e^{\lambda(t - \tau)} \left[\lambda(t - \tau) + 1\right] M^{-1} \vH_{i}(\tau)d\tau,
	\end{eqnarray}
	where $t_{0} = t_{0}(\vW(0), \varepsilon_{0}) \ge 0$ is such that $\abs{\v{H}_{i}(t)} \le C_{H} (\varepsilon_{0})$, for all $i = 1, \ldots, N$ and all $t \ge t_{0}$, with $C_{H} (\varepsilon_{0})$ is as defined by \eqref{non-diss:eq:c-h}.
	
	Since $\lambda < 0$, the first terms in \eqref{non-diss:eq:bdd-sols-duhamel-pos-repeated} and \eqref{non-diss:eq:bdd-sols-duhamel-vel-repeated} vanish as $t \to \infty$, and therefore the asymptotic behavior of $\vx_{i}(t)$ and $\vw_{i}(t)$ is determined by the corresponding integral terms.
	
	For $t \ge t_{0}$, for the integral term in \eqref{non-diss:eq:bdd-sols-duhamel-pos}, we have
	\begin{equation*}
		\abs{ 
			\int_{t_{0}}^{t} e^{\lambda(t - \tau)} (t - \tau) M^{-1} \vH_{i}(\tau)d\tau
		}
		\le 
		\left[ e^{\lambda (t - t_{0})} \left(\frac{1}{\lambda} (t - t_{0}) - \frac{1}{\lambda^{2}}\right) + \frac{1}{\lambda^{2}}\right] M^{-1} C_{H}(\varepsilon_{0}) 
	\end{equation*}
	and for the integral term in \eqref{non-diss:eq:bdd-sols-duhamel-vel}, we have
	\begin{equation*}
		\abs{ 
			\int_{t_{0}}^{t}  e^{\lambda(t - \tau)} \left[\lambda(t - \tau) + 1\right] M^{-1} \vH_{i}(\tau)d\tau
		}
		\le 
		e^{\lambda (t - t_{0})} (t - t_{0}) M^{-1} C_{H}(\varepsilon_{0}).
	\end{equation*}
	Then (i) and (ii) follow.
\end{proof}

The results of the above theorems are summarized in the following corollary.
\begin{corollary}
	Suppose that Assumptions \ref{non-diss:assmp:fl-is-bdd} and \ref{non-diss:assmp:sp-forces-are-approx-linear} hold. 
	Then the group of agents exhibits approximate flocking.
	If, in addition, $\tilde{K}_{v}^{2} - 4\tilde{K}_{p} = 0$, the flocking is proper and exact.
\end{corollary}

\begin{remark}
	\label{non-diss:rmrk:random}
	\normalfont
	Suppose that one wants to incorporate stochastic effects into the model by adding random forces $\vf_{i}^{R}$, $i = 1, \ldots, N$, to \eqref{model:dynamics}.
	If $\vf_{i}^{R}(t)$, $t \ge 0$, viewed as a realization of some stochastic process, is always $C^{1}$ and bounded, then Theorems \ref{non-diss:thm:bdd-sols} and \ref{non-diss:thm:bdd-sols-repeated} will still hold.
	This can be easily seen from the fact that $\vf^{L}$ was merely required to be bounded (in addition to the initial assumptions of being $C^{1}$) in the proofs of either of the above theorems.
	Thus, if another bounded term $\vf_{i}^{R}$ is added to the right-hand side of \eqref{model:dynamics}, Theorem \ref{non-diss:thm:bdd-sols} will still guarantee that the solutions are (almost surely) bounded.
	Similarly, for the case of repeated eigenvalue, the system will still (almost surely) exhibit exact proper flocking due to Theorem  \ref{non-diss:thm:bdd-sols-repeated}.
\end{remark}

\begin{remark}
	\label{non-diss:rmrk:global}
	\normalfont
	Let $\eta > 0$.
	It follows from Theorem \ref{non-diss:thm:bdd-sols} that, for $\tilde{K}_{v}^{2} - 4 \tilde{K}_{p}^{2} \ne 0$, there exists a set
	\begin{equation*}
		\mathcal{B} = 
		\left\{
			(\vX, \vW) \in \R^{2Nd} : 
			\abs{\vx_{i}} \le \frac{2 (1 + \eta) \tilde{K}_{v}}{\tilde{K_{p}} \sqrt{\abs{\tilde{K}_{v}^{2} - 4 \tilde{K}_{p}^{2}}}} C_{2},
			\;\;
			\abs{\vw_{i}} \le \frac{2 (1 + \eta)}{\sqrt{\abs{\tilde{K}_{v}^{2} - 4 \tilde{K}_{p}^{2}}}} C_{2},
			\;\;
			1 \le i \le N
		\right\},
	\end{equation*} 
	such that, for any bounded set $\mathcal{B}_{0} \subseteq \R^{2Nd}$, an orbit starting in $\mathcal{B}_{0}$ enters $\mathcal{B}$ after some time $t_{1} = t_{1}(\mathcal{B}_{0})$.
	In other words, $\mathcal{B}$ is a global absorbing set for \eqref{diss:dynamics}.
	Since $\mathcal{B}$ is bounded, \eqref{diss:dynamics} possesses a global (compact) attractor (see, e.g., \cite{temam1997infinite}).
	Similarly, for $\tilde{K}_{v}^{2} - 4 \tilde{K}_{p}^{2} = 0$, Theorem \ref{non-diss:thm:bdd-sols-repeated} guarantees the existence of a global bounded absorbing set 
	\begin{equation*}
		\mathcal{B} = 
		\left\{
		(\vX, \vW) \in \R^{2Nd} : 
		\abs{\vx_{i}} \le \frac{4(1 + \eta)}{\tilde{K_{v}}} C_{2},
		\;\;
		\abs{\vw_{i}} \le \eta,
		\;\;
		i = 1, \ldots, N
		\right\}.
	\end{equation*}
	Thus, \eqref{diss:dynamics} possesses a global attractor for any configuration of the control parameters $\tilde{K}_{p}$ and $\tilde{K}_{v}$.
\end{remark}

\subsection{Exact flocking}

Suppose that $\tilde{K}_{v}^{2} - 4\tilde{K}_{p} \ne 0$. To show that the system exhibits exact flocking in this case, introduce a new coordinate system centered at the center of mass of the group $(\bvx, \bvw) = \left(\sum_{i = 1}^{N} \bvx_{i}, \sum_{i = 1}^{N} \bvx_{i} \right)$ by defining
\begin{equation*}
	\hvx_{i} = \vx_{i} - \bvx, 
	\;\;
	\hvw_{i} = \vw_{i} - \bvw,
	\;\;
	i = 1, \ldots, N,
\end{equation*}
and 
\begin{equation*}
	\hvX = (\hvx_{1}, \ldots \hvx_{N}),
	\;\;
	\hvW = (\hvw_{1}, \ldots \hvw_{N}).
\end{equation*}
As before, we have $\hvx_{i} - \hvx_{j} = \vx_{i} - \vx_{j}$ and $\hvw_{i} - \hvw_{j} = \vw_{i} - \vw_{j}$ for $i, j = 1, \ldots, N$.
Then the dynamics of the center of mass is given by
\begin{equation}
	\label{non-diss:eq:dynmaics-com}
	\begin{aligned}
		\dot{\bvx} & = \bvw,
		\\
		M \dot{\bvw} & = \frac{1}{N} \sum_{i = 1}^{N} \vu^{P} (\bvx + \hvx_{i}) + \frac{1}{N} \sum_{i = 1}^{N} \vu^{V} (\bvw + \hvw_{i}) -  M \vf^{L}
	\end{aligned}
\end{equation}
and the dynamics of the agents relative to the center of mass is given by
\begin{equation}
	\label{non-diss:eq:dynmaics-hat}
	\begin{aligned}
		\dot{\hvx}_{i} & = \hvw_{i},
		\\
		M \dot{\hvw}_{i} & 
		= \vf_{i}^{C} + \vf_{i}^{D} 
		+ \vu^{P} (\bvx + \hvx_{i}) + \vu^{V} (\bvw + \hvw_{i})
		\\
		& \phantom{=}- \frac{1}{N} \sum_{j = 1}^{N} \vu^{P} (\bvx + \hvx_{j})
		- \frac{1}{N} \sum_{j = 1}^{N} \vu^{V} (\bvw + \hvw_{j}),
	\end{aligned}
	\;\; i = 1, \ldots, N.
\end{equation}

We now make the assumption that the activation threshold $r_{0}$ of the position alignment force is sufficiently large.
\begin{assumption}
	\label{non-diss:assmp:r0-lb}
	\begin{equation*}
		r_{0} > \frac{2 \tilde{K_{v}}}{\tilde{K_{p}} \sqrt{\abs{\tilde{K_{v}}^{2} - 4\tilde{K}_{p}}}} C_{2}.
	\end{equation*}
\end{assumption}
The above assumption, along with Theorem \ref{non-diss:thm:bdd-sols} (i), imply that there exists $t_{1}$ such that 
\begin{equation}
	\label{non-diss:eq:up-vanish}
	\abs{\vx_{i}(t)} \le r_{0}, 
	\;\;
	i = 1, \ldots, N,
\end{equation}
for all $t \ge t_{1}$, and therefore the position alignment force vanishes for $t \ge t_{1}$.
Suppose that $t_{1}$ is the smallest time such that \eqref{non-diss:eq:up-vanish} holds for all $t \ge t_{1}$.
In this case, we refer to $t_{1}$ as the \textit{spatial stabilization time}.

Let $D \delta \vu^{V}(\vy)$ denote the Jacobian of $\delta \vu^{V}$ at $\vy \in \R^{d}$.
We now impose an additional assumption on the nonlinear term of the velocity alignment force.
\begin{assumption}
	\label{non-diss:assmp:sp-forces-jacobian-is-bdd}
	There exists $C_{\delta V}$ such that $\abs{D \delta \vu^{V} (\vy)} \le C_{\delta V}$ for all $\vy \in \R^{d}$.
\end{assumption}
Note that, whenever Assumptions \ref{non-diss:assmp:fl-is-bdd} and \ref{non-diss:assmp:sp-forces-are-approx-linear} hold, velocities of the agents are bounded, so $\delta \vu^{V}$ takes values from a compact subset of $\R^{d}$. 
Then $\abs{D \delta \vu^{V} (\vw_{i}(t))}$ will be bounded for $i = 1, \ldots, N$ and all $t \ge 0$, provided $\vu^{V}$ is $C^{1}$.

First, consider the dynamics of the agents relative to the center of mass. 
For $t \ge t_{1}$, the second equation in \eqref{non-diss:eq:dynmaics-hat} becomes
\begin{equation*}
	\begin{aligned}
		M \dot{\hvw}_{i} & 
		= \vf_{i}^{C} 
		+ \vf_{i}^{D} 
		- K_{v} \hvw_{i}
		+ \delta \vu^{V} (\bvw + \hvw_{i})
		- \frac{1}{N} \sum_{j = 1}^{N} \delta \vu^{V} (\bvw + \hvw_{j}),
	\end{aligned}
	\quad 
	i = 1, \ldots, N.
\end{equation*}
Using a first-order Taylor expansion of $\delta \vu^{V}$ at $\bvw$, we get
\begin{equation*}
	\begin{aligned}
		M \dot{\hvw}_{i} & 
		= \vf_{i}^{C} + \vf_{i}^{D} 
		- K_{v} \hvw_{i}
		+ D \delta \vu^{V} (\vzeta_{i}) \hvw_{i}
		- \frac{1}{N} \sum_{j = 1}^{N} D \delta \vu^{V} (\vzeta_{j}) \hvw_{j},
	\end{aligned}
	\;\;
	i = 1, \ldots, N,
\end{equation*}
for all $t \ge t_{1}$, where $\vzeta_{i} \in \R^{d}$ belongs to the line segment connecting $\bvw$ and $\hvw_{i}$.
Then, for $t \ge t_{1}$, \eqref{non-diss:eq:dynmaics-hat} can be written as
\begin{equation}
	\label{non-diss:eq:dynamics-hat-noconserve}
	\begin{aligned}
		\dot{\hvX} & = \hvW,
		\\
		M \dot{\hvW} & 
		= - \nabla_{\hvX} U(\hvX) - L(\hvX) \hvW - K_{v} \hvW
		+ \Delta(t) \hvW
	\end{aligned}
\end{equation}
where $\Delta(t)$ is a block matrix with $N \times N$ blocks of size $d \times d$, where the $(i, j)$-th block is given by
\begin{equation*}
	\left[\Delta(t)\right]_{i, j}
	=
	\begin{cases}
		(1 - \frac{1}{N}) D \delta \vu^{V} (\vzeta_{i}(t)), & \text{if} \ i = j, 
		\\
		\frac{1}{N} D \delta \vu^{V} (\vzeta_{j}(t)), & \text{otherwise}.
	\end{cases} 
\end{equation*}

Note that Assumption \ref{non-diss:assmp:sp-forces-jacobian-is-bdd} implies that
\begin{equation*}
	\abs{\Delta(t)} \le C_{3}
\end{equation*}
for all $t \ge t_{1}$, where
\begin{equation*}
	C_{3}  = \left[2(N-1)\right]^{\frac{1}{2}} C_{\delta V}.
\end{equation*}
	
As before, for the reduced system
\begin{equation}
	\label{non-diss:eq:dynamics-hat-noconserve-reduced}
	\begin{aligned}
		\dot{\hvX} & = \hvW,
		\\
		M \dot{\hvW} & 
		= - \nabla_{\hvX} U(\hvX) - L(\hvX) \hvW - K_{v} \hvW,
	\end{aligned}
\end{equation}
we can show that the velocities decay exponentially.
\begin{lemma}
	\label{non-diss:lemma:w-hat-reduced-decay}
	Let $(\hvX(t), \hvW(t))$ be a solution of the reduced system \eqref{non-diss:eq:dynamics-hat-noconserve-reduced}.
	Then 
	\begin{equation*}
		\babs{\hvW(t)} \le \babs{\hvW(t_{1})} e^{-\frac{K_{v}}{M}(t-t_{1})}
	\end{equation*}
	for all $t \ge t_{1}$.	
\end{lemma}
\begin{proof}
	The proof is identical to that of Lemma \ref{non-diss:lemma:w-reduced-decay} if one replaces \eqref{non-diss:eq:energy-reduced-def} with
	\begin{equation*}
		E_{r}(\hvX, \hvW) = \frac{1}{2} M \babs{\hvW}^{2} + \vU(\hvX).
	\end{equation*}
\end{proof}

We can now prove another key result of this section.
\begin{theorem}
	\label{non-diss:thm:flocking}
	Let $(\hvX(t), \hvW(t))$ be a solution of \eqref{non-diss:eq:dynamics-hat-noconserve}.
	Suppose that Assumption \ref{non-diss:assmp:sp-forces-jacobian-is-bdd} hold.
	Then, for $i = 1, \ldots, N$, 
	\begin{enumerate}[(i)]
		\item $\abs{\hvx_{i}(t) - \hvx_{i}(t_{1})} \le \babs{\hvW(t_{1})} \left[ \frac{M}{2\hat{K}} \left(1 - e^{-\frac{2\hat{K}}{M} (t - t_{1})}\right) \right]^{\frac{1}{2}}$;
		\item $\abs{\hvw_{i}(t)} \le \babs{\hvW(t_{1})} e^{-\frac{\hat{K}}{M}(t-t_{1})}$;
	\end{enumerate}
	for all $t \ge t_{1}$, where 
	\begin{equation*}
		\hat{K} = K_{v} - C_{3}.
	\end{equation*}
\end{theorem}
\begin{proof}
	Let $\hat{\Psi}(t, s)$, $t \ge s \ge 0$, be the state-transition matrix of the linear system \eqref{non-diss:eq:dynamics-hat-noconserve-reduced}.
	By the variation of constants formula, the solution of \eqref{non-diss:eq:dynamics-hat-noconserve} is given by
	\begin{equation*}
		\begin{aligned}
			(\hvX(t), \hvW(t))
			& = \hat{\Psi}(t, t_{1}) (\hvX(t_{1}), \hvW(t_{1}))
			+ \int_{t_{1}}^{t} \hat{\Psi}(t, \tau) (\vzero, M^{-1}\Delta(\tau)\hvW(\tau))  d\tau.
		\end{aligned}
	\end{equation*}
	As before, let $\hat{\Psi}_{\vW}(t, s)$ be the last $Nd$ rows of the matrix $\hat{\Psi}(t, s)$.
	Using Lemma \ref{non-diss:lemma:w-hat-reduced-decay}, we get
	\begin{equation}
		\label{non-diss:eq:blah-blah}
		\begin{aligned}
			\babs{\hvW(t)}
			& \le \babs{\hat{\Psi}_{\vW}(t, t_{1}) (\hvX(t_{1}), \hvW(t_{1}))}
			+ \abs{\int_{t_{1}}^{t} \hat{\Psi}_{\vW}(t, \tau)  (\vzero, M^{-1}\Delta(\tau)\hvW(\tau))   d\tau} 
			\\
			& \le \babs{\hvW(t_{1})} e^{-\frac{K_{v}}{M}(t-t_{1})}
			+ \int_{t_{1}}^{t} M^{-1} C_{3} \babs{\hvW(\tau)} e^{-\frac{K_{v}}{M}(t-t_{1})}  d\tau.
		\end{aligned}
	\end{equation}
	Let $\Lambda(t) = \babs{\hvW(t)} e^{\frac{K_{v}}{M}t}$. 
	Then \eqref{non-diss:eq:blah-blah} can be written as
	\begin{equation*}
		\Lambda(t) \le \Lambda(t_{1}) + \int_{t_{1}}^{t} M^{-1} C_{3} \Lambda(\tau) d\tau.
	\end{equation*}
	By Gr\"{o}nwall's inequality, we get
	\begin{equation*}
		\Lambda(t) \le \Lambda(t_{1}) \exp \left(\int_{t_{1}}^{t} M^{-1} C_{3} d\tau \right).
	\end{equation*}
	Thus,
	\begin{equation*}
		\babs{\hvW(t)} \le \babs{\hvW(t_{1})} \exp\left(-\frac{1}{M} \left(K_{v} - C_{3}\right) (t-t_{1})\right),
	\end{equation*}
	and (ii) follows.
	
	Now, using Jensen's inequality, for $i = 1, \ldots, N$, we get
	\begin{equation}
		\begin{aligned}
			\abs{\hvx_{i}(t) - \hvx_{i}(t_{1})}^{2}
			& = \sum_{m = 1}^{d} \abs{\hvx_{i}^{(m)}(t) - \hvx_{i}^{(m)}(t_{1})}^{2}
			\\
			& = \sum_{m = 1}^{d} \abs{\int_{t_{1}}^{t} \hvw_{i}^{(m)}(\tau) d \tau }^{2}
			\\
			& \le \int_{t_{1}}^{t} \abs{\hvw_{i}(\tau)}^{2} d \tau 
			\\
			& \le \babs{\hvW(t_{1})}^{2} \int_{t_{1}}^{t} e^{-\frac{2\hat{K}}{M} (\tau - t_{1})} d \tau 
			\\
			& \le \babs{\hvW(t_{1})}^{2} \frac{M}{2\hat{K}} \left(1 - e^{-\frac{2\hat{K}}{M} (t - t_{1})}\right),
		\end{aligned}
	\end{equation}
	and (i) follows.
\end{proof}

\begin{corollary}
	\label{non-diss:corr:flocking}
	Suppose that Assumptions \ref{non-diss:assmp:fl-is-bdd}-\ref{non-diss:assmp:sp-forces-jacobian-is-bdd} hold and that $\hat{K} > 0$.
	Then 
	\begin{enumerate}[(i)]
		\item after the spatial stabilization time, the agents' positions will not change by more than $\babs{\hvW(t_{1})} \left(\frac{M}{\hat{K}} \right)^{\frac{1}{2}}$;
		\item $\abs{\vw_{i}(t) - \bvw(t)} \to 0$ as $t \to \infty$, for all $i = 1, \ldots, N$, and the group exhibits exact flocking;
	\end{enumerate}
\end{corollary}
\begin{proof}
	Since $\hat{K} > 0$, (i) follows from Theorem \ref{non-diss:thm:flocking} (i).
	(ii) follows from Theorem \ref{non-diss:thm:flocking} (ii) and Theorem \ref{non-diss:thm:bdd-sols} (i).
\end{proof}

\begin{remark}
	\normalfont
	The rate of convergence to the velocity consensus as well as spatial stability of the group depend on $\hat{K}$, which, in turn, depends on the tunable parameters $K_{v}$ and $C_{\delta V}$. 
	Informally speaking, the former represents the magnitude of the linear component of the velocity alignment force, and the latter represents the magnitude of its nonlinear part.
	The stronger the linear component of the force, the faster the group will converge to a velocity consensus and the more rigid spatial configuration it will maintain.
\end{remark}

\subsection{Dynamics of the center of mass}

Now, consider the dynamics of the center of mass. 
For $t \ge t_{1}$, \eqref{non-diss:eq:dynmaics-com} can be written as
\begin{equation}
	\label{non-diss:eq:dynmaics-com-noconserve}
	\begin{aligned}
		\dot{\bvx} & = \bvw,
		\\
		M \dot{\bvw} & = -K_{v} \bvw + \frac{1}{N} \sum_{i = 1}^{N} \delta \vu^{V} (\bvw + \hvw_{i}) -  M\vf^{L}.
	\end{aligned}
\end{equation}
We show that the deviation of the velocity of the center of mass from that of the virtual leader stays bounded after the spatial stabilization time.
\begin{theorem}
	\label{non-diss:thm:com}
	Let $(\bvx(t), \bvw(t))$ be a solution of \eqref{non-diss:eq:dynmaics-com-noconserve}.
	Suppose that Assumption \ref{non-diss:assmp:sp-forces-jacobian-is-bdd} hold.
	Then, for any $\varepsilon_{2} > 0$, there exists $t_{2} = t_{2}(\bvw(t_{1}), \varepsilon_{2})$ such that 
	\begin{equation}
		\label{non-diss:eq:com-upper-bound}
		\abs{\bvw(t)} \le \varepsilon_{2} + \frac{2 C_{4}}{K_{v}}
	\end{equation}
	for all $t \ge t_{2}$, where
	\begin{equation*}
		C_{4} = N^\frac{1}{2} C_{\delta V} + MC_{l}.
	\end{equation*}
\end{theorem}
\begin{proof}
	Note that \eqref{non-diss:eq:dynmaics-com-noconserve} is decoupled, so in order to obtain \eqref{non-diss:eq:com-upper-bound} it suffices to consider only the second equation in \eqref{non-diss:eq:dynmaics-com-noconserve}.
	As before, a solution $\bvw_{r}(t)$ of the linear reduced system 
	\begin{equation}
		\label{non-diss:eq:dynmaics-com-noconserve-reduced}
		M \dot{\bvw}_{r} = -K_{v} \bvw_{r}
	\end{equation}
	will satisfy
	\begin{equation*}
		\abs{\bvw_{r}(t)} \le \abs{\bvw_{r}(t_{1})} e^{-\frac{K_{v}}{M} (t-t_{1})}
	\end{equation*}
	for all $t \ge t_{1}$.
	Let $\Psi_{\bvw}(t, s)$, $t \ge s \ge 0$, be the state-transition matrix of \eqref{non-diss:eq:dynmaics-com-noconserve-reduced}.
	Then, by the variation of constants formula, we get
	\begin{equation*}
		\begin{aligned}
			\abs{\bvw(t)} 
			& =
			\abs{\Psi_{\bvw}(t, t_{1}) \bvw(t_{1})
			+ \int_{t_{1}}^{t} \Psi_{\bvw}(t, \tau) \left( \frac{1}{NM} \sum_{i = 1}^{N} \delta \vu^{V} (\bvw + \hvw_{i}) - \vf^{L}(\tau) \right) d \tau }
			\\
			& \le 
			\abs{\bvw_{r}(t_{1})} e^{-\frac{K_{v}}{M} (t-t_{1})}
			+ \frac{1}{M} \int_{t_{1}}^{t} \abs{ \frac{1}{N} \sum_{i = 1}^{N} \delta \vu^{V} (\bvw + \hvw_{i}) - M \vf^{L}(\tau)} e^{-\frac{K_{v}}{M} (t-t_{1})} d \tau
			\\
			& \le 
			\left(\abs{\bvw_{r}(t_{1})} - \frac{C_{4}}{K_{v}} \right) e^{-\frac{K_{v}}{M} (t-t_{1})} + \frac{C_{4}}{K_{v}}.
		\end{aligned}
	\end{equation*}
	Since $K_{v} > 0$, the result follows.
\end{proof}

\begin{remark}
	\normalfont
	For the case $\tilde{K}_{v}^{2} - 4\tilde{K}_{p} \ne 0$, for large times, Corollary \ref{non-diss:corr:flocking} together with Theorem \ref{non-diss:thm:com} can potentially provide a tighter bound on $\abs{\vw_{i}(t)}$, $i = 1, \ldots, N$, than the one suggested in Theorem \ref{non-diss:thm:bdd-sols} (ii).
	However, which bound is tighter would depend on the relation between the tunable parameters $K_{p}, K_{v}, C_{p}, C_{l}$, and $C_{\delta V}$.	
\end{remark}

%% file: sections/06_computational.tex
\section{Computational examples}
\label{section:comp}

In this section, we first present results of numerical simulations demonstrating that the system exhibits flocking dynamics in both dissipative and non-dissipative cases.
We discuss qualitatively distinct motion regimes that can be achieved using different combinations of the model parameters.
Finally, we provide some numerical illustrations of wobblers, the non-equilibrium solutions that can be a part of the attractor in the non-strictly dissipative case.

\subsection{Non-dimensional units}
For numerical simulations, it is convenient to introduce non-dimensional units, allowing to work with numerical values of order of unity and obtain some characteristic quantities defining different motion regimes.
Let
\begin{equation*}
	m^{\prime} = \frac{[m]}{M},
	\;\;
	r^{\prime} = \frac{[r]}{r_{C}},
	\;\;
	v^{\prime} = \frac{[v]}{v_{char}},
	\;\;
	t^{\prime} = \frac{[t]}{r_{C} / v_{char}},
\end{equation*}
where $v_{char}$ is the characteristic speed of the virtual leader (and thus the characteristic speed of all the agents in the group, assuming that the agents are following the leader's trajectory).
Substituting the above into \eqref{model:dynamics}, we get
\begin{equation}
	\label{comp:dynamics-nondim}
	\begin{aligned}
		\dot{\vq}_{i}^{\prime} & = \vv_{i}^{\prime},
		\\
		\dot{\vv}_{i}^{\prime} & = 
			A^{\prime} \sum_{i \ne j} w_{C} \left(r_{C} \abs{\vq_{ij}^{\prime}}\right) \vq_{ij}^{\prime} 
			- B^{\prime} \sum_{i \ne j} w_{D} \left(r_{D} \abs{\vq_{ij}^{\prime}}\right) \vv_{ij}^{\prime} 
		\\
			& \phantom{12345}
			- \alpha^{\prime} k \left(r_{C} \abs{\vq_{il}^{\prime}} \right) \frac{\vq_{il}^{\prime}}{\abs{\vq_{il}^{\prime}}}
			- \beta^{\prime} p \left(v_{char} \abs{\vv_{il}^{\prime}} \right) \frac{\vv_{il}^{\prime}}{\abs{\vv_{il}^{\prime}}},
	\end{aligned}
	\qquad	i = 1, \ldots, N,
\end{equation}
where $\vq^{\prime}_{i} = r_{C}^{-1} \vq_{i}$, $\vv^{\prime}_{i} = v_{char}^{-1} \vv_{i}$, for $i = 1, \ldots, N + 1$.
The dimensionless parameters
\begin{equation}
	\label{comp:eq:coeff}
	A^{\prime} = \frac{r_{C}^{2}}{M v_{char}^{2}} A,
	\;\;
	B^{\prime} = \frac{r_{C}}{M v_{char}} B,
	\;\;
	\alpha^{\prime} = \frac{r_{C}^{2}}{M v_{char}^{2}} \alpha,
	\;\; 
	\beta^{\prime} = \frac{r_{C}}{M v_{char}} \beta 
\end{equation}
represent relative contributions of the corresponding forces to the dynamics of an agent.
The parameters
\begin{equation*}
	r_{D}^{\prime} = \frac{r_{D}}{r_{C}},
	\;\;
	r_{0}^{\prime} = \frac{r_{0}}{r_{C}}
\end{equation*}
represent the cut-off distance of the ambient dissipative force and the activation threshold of the position alignment force, respectively, measured in terms of the cut-off distance of the ambient conservative force.
The parameter
\begin{equation*}
	v_{0}^{\prime} = \frac{v_{0}}{v_{char}}
\end{equation*}
represents the activation threshold of the velocity alignment force measured in terms of the characteristic speed of the virtual leader. 

\begin{figure*}[h]
	\centering
	\begin{subfigure}[t]{0.4\linewidth}
		\includegraphics[width=1\textwidth]{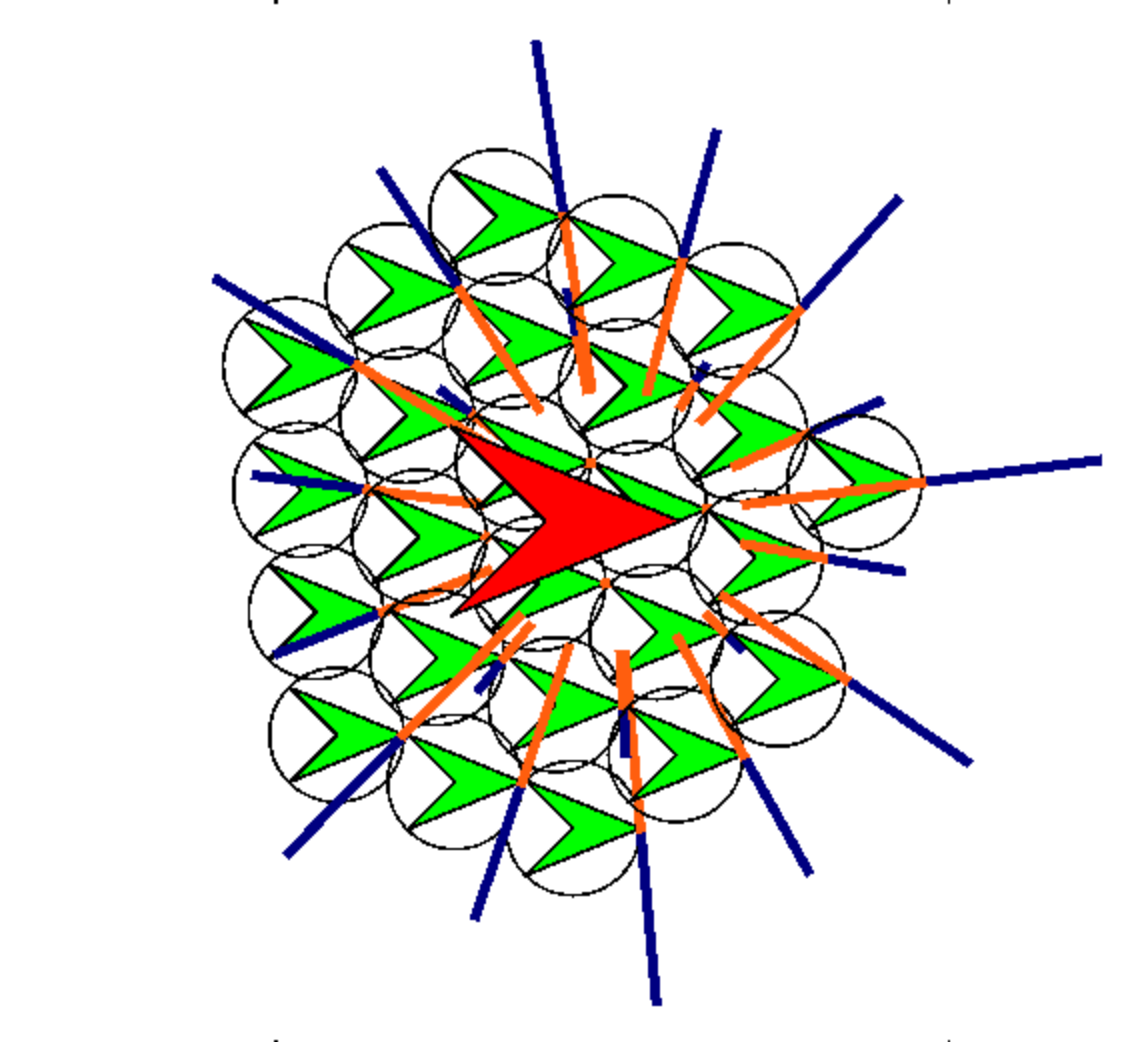}
		\caption{Tight packing}
	\end{subfigure}
	\begin{subfigure}[t]{0.4\linewidth}
		\includegraphics[width=1\textwidth]{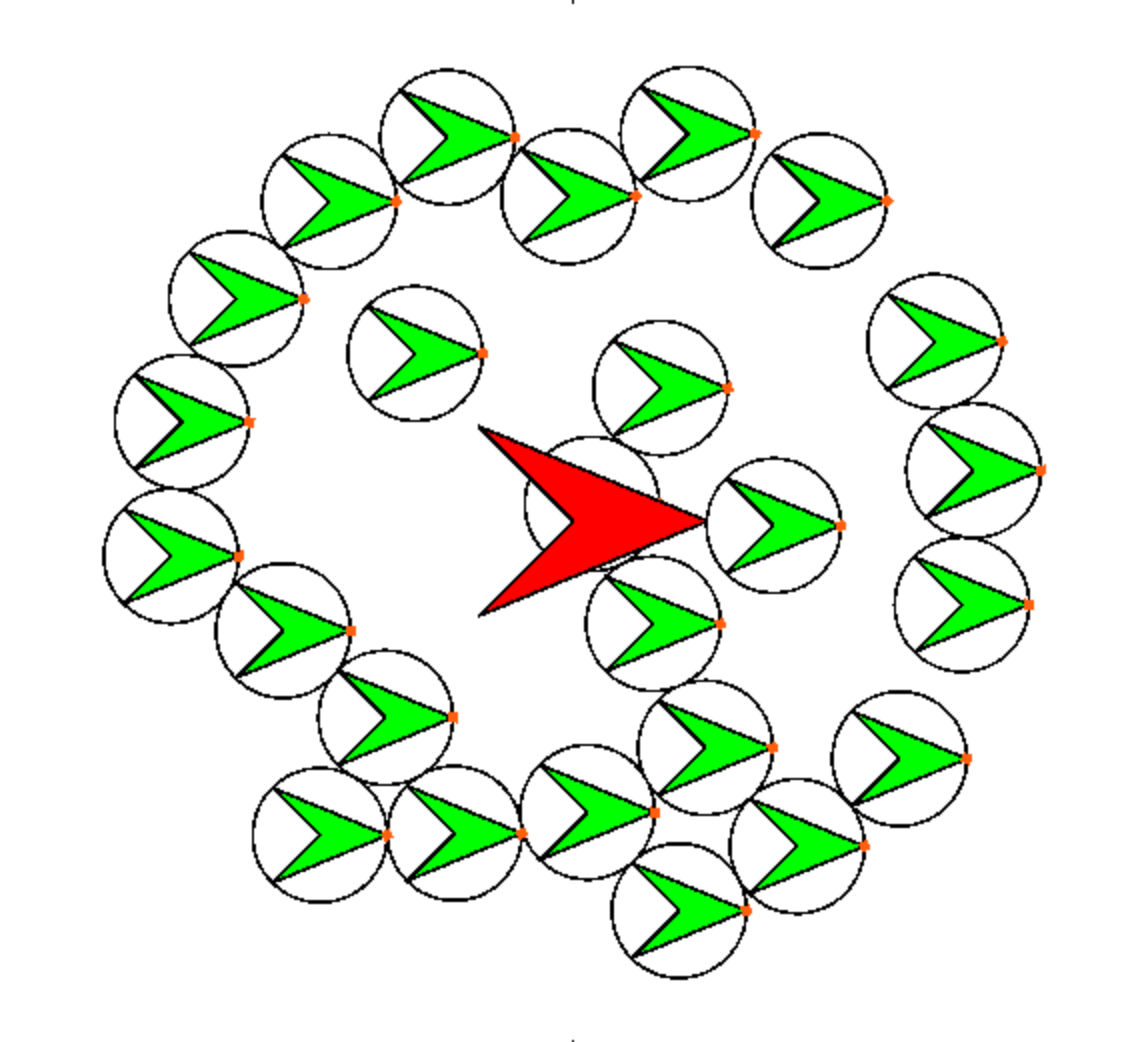}
		\caption{Loose packing}
	\end{subfigure}
	\caption{A stabilized group of \(N = 25\) agents moving in a two-dimensional space. (a) illustrates a ``tight packing'' configuration with a high group density \(\rho\), where the followers (small green arrows) are highly concentrated around the virtual leader (large red arrow). Both ambient conservative forces (blue lines attached to the arrow tips) and self-propulsion conservative forces (orange lines attached to the arrow tips) are non-zero for the agents that are farther than \(r_0\) from the virtual leader. (b) illustrates a ``loose packing'' configuration with a low group density \(\rho\). In this case, all conservative forces are zero. The radius of the circles enclosing the followers is equal to the cut-off distance  $r_{c}$ of the conservative force. (Color figure online)}
	\label{comp:fig:density}
\end{figure*}

Using the non-dimensional units introduced above, we set $A^{\prime} = 10$ and $r_{D}^{\prime} = 5$.
The remaining dimensionless parameters $B^{\prime}$, $\alpha^{\prime}$, $r_{0}^{\prime}$, $\beta^{\prime}$, and $v_{0}^{\prime}$ are assumed to be free.
We use velocity-Verlet \cite{swope1982computer} integration scheme, implemented in a custom simulator, setting the time step to $\Delta t = 10^{-2}$ and running each simulation for $n = 10^5$ steps, which is sufficient for the system to stabilize.

\subsection{Motion regimes} 

Recall that the position alignment force pulls an agent towards the virtual leader whenever the distance between the two becomes less than $r_{0}^{\prime}$, hence making the group tend to stay within the ball $B(\vq_{l}^{\prime}(t^{\prime}), r_{0}^{\prime})$.
Consequently, the quantity
\begin{equation*}
	\rho = \frac{N V_{d}(0.5)}{V_{d}(r_{0}^{\prime})},
\end{equation*}
where $V_{d}(r)$ is the volume of a $d$-dimensional ball of radius $r > 0$, can be seen as the ``density'' of the group.
Higher values of $\rho$ correspond to ``tighter packing'', characterized by a stronger influence of both ambient and self-propulsion conservative forces. 
Conversely, lower values of $\rho$ correspond to ``looser packing'', where the impact of conservative forces is weaker allowing the system to have some ``breathing room''.
Figure \ref{comp:fig:density} provides an illustration of the two scenarios.

The quantity 
\begin{equation*}
	v_{b} = \frac{v_{0}}{r_{C}}
\end{equation*}
represents the characteristic speed of the virtual leader measured in body lengths per second (blps) \cite{sitti2017mobile}. 
It follows from \eqref{comp:eq:coeff} that $A^{\prime} / B^{\prime}  = v_{b}^{-1} A / B$.
Here, $A / B$ is a fixed quantity that depends only on the properties of the material of the elastic bumpers and the properties of the medium in which the group operates.
Thus, the characteristic speed of the virtual leader can be implicitly set by setting the ratio $A^{\prime} / B^{\prime}$.
In our simulations, we set $A / B = 10$.

By varying the free parameters, one can obtain different combinations of \(\rho\) and \(v_{b}\), leading to qualitatively distinct motion regimes.
We consider nine such motion regimes, as listed in Table \ref{comp:tab:regimes}.
\begin{table}[h!]
	\centering
	\begin{tabular}{| c | c | c | c | c |c | c | c |} 
		\hline
		\multirow{2}{*}{\textbf{Regime}} & \multicolumn{5}{c|}{\textbf{Parameters}} & \multirow{2}{*}{$\bm{\rho}$} &  \multirow{2}{*}{$\bm{v_{b}}$ } \\ 		
		[0.5ex] 
		\cline{2-6}
		& $B^{\prime}$ & $\alpha^{\prime}$ & $r_{0}^{\prime}$ & $\beta^{\prime}$ & $v_{0}^{\prime}$ &  & \\ 
		\hline
		(1) typical                  &   1 & 1 & 4.64 &   1 & 0.5 & 0.25 &   1 \\
		\hline
		(2) no velocity alignment    &   1 & 1 & 4.64 &   0 &   0 & 0.25 &   1 \\
		(3) tight velocity alignment &   1 & 1 & 4.64 &   1 &   0 & 0.25 &   1 \\
		(4) no position alignment    &   1 & 0 & 4.64 &   1 &   0 & 0.25 &   1 \\
		\hline
		(5) low density              &  10 & 1 & 7.93 &   1 & 0.5 &  0.1 &   1 \\
		(6) optimal density          &  10 & 1 & 3.23 &   1 & 0.5 & 0.74 &   1 \\
		(7) high density             &  10 & 1 & 2.32 &   1 & 0.5 &   2  &   1 \\
		\hline
		(8) low speed                & 0.2 & 1 & 4.64 & 0.2 & 0.5 & 0.25 & 0.2 \\
		(9) high speed               &   5 & 1 & 4.64 &   5 & 0.5 & 0.25 &   2 \\
		\hline
	\end{tabular}
	\caption{Motion regimes and corresponding values of the free parameters for a group of $N = 100$ agents in a $3$-dimensional space with $A / B = 10$.}
	\label{comp:tab:regimes}
\end{table} 

Regime (1) is characterized by the presence of both position and velocity alignment forces, as well as a moderate group density.
In this regime, velocity alignment is not ``tight'', by which we mean that the corresponding force is activated only when an agent's velocity deviation from that of the virtual leader is sufficiently large.
We assume that regime (1) is typical.
In regimes (2)-(4), we disable one of the self-propulsion forces and ``tighten'' the velocity alignment.
In regimes (5)-(7), we vary the density of the group while keeping the rest of the parameters fixed.
Note that we refer to the density $\rho = 0.74$ as optimal since it is approximately equal to the maximum density for the packing of a collection of identical balls in $\R^{3}$ \cite{conway2013sphere}.
For $\rho < 0.74$, it is guaranteed that the ambient conservative force is non-zero for at least one pair of agents.
Finally, in regimes (8)-(9), we vary the characteristic speed of the group.

We set $k(s) = p(s) = s$ in \eqref{model:eq:h-s-def}, so that the position and the velocity alignment forces become piecewise linear.
Although in this setting $g$ and $h$ fail to be $C^{1}$, we assume that they represent numerical approximations of some $C^{1}$ functions that differ from $g$ and $h$ only on small intervals containing the points of non-differentiability.

\subsection{Flocking}

As before, we distinguish two possible scenarios: dissipative and non-dissipative.
To simulate the dissipative scenario, we let the virtual leader move along a straight line at the constant speed $v_{char}$.
For the non-dissipative scenario, the leader moves along a circle of radius $10r_{0}^{\prime}$ at the constant speed $v_{char}$.
The simulations are performed for $d = 3$ and $N = 100$.
The initial positions of the agents are randomly sampled from the uniform distribution on $B(\vq_{l}^{\prime}(0), r_{init}^{\prime})$, where $r_{init}^{\prime} = 10$, resulting in an initial spatial configuration with a group density of $\rho = 0.025$.
The initial velocities are randomly sampled from the $d$-variate truncated normal distribution with mean $\vv_{l}^{\prime}({0})$, covariance matrix $I_{d}$, and support $B(\vv_{l}^{\prime}({0}), 1)$.

In the process of simulations, we measure the following quantities:
\begin{equation*}
	\begin{gathered}
		\bar{q}_{dev}(t^{\prime}) = \frac{1}{N} \sum_{i = 1}^{N}\abs{\vq_{i}^{\prime}(t^{\prime}) - \vq_{l}^{\prime}(t^{\prime})},
		\\
		\bar{v}_{dev}(t^{\prime}) = \frac{1}{N} \sum_{i = 1}^{N} \abs{\vv_{i}^{\prime}(t^{\prime}) - \vv_{l}^{\prime}(t^{\prime})},
	\end{gathered}
\end{equation*}
and
\begin{equation*}
	\begin{aligned}
		\bar{U}(t^{\prime}) 
		&= \frac{1}{N} \sum_{i = 1}^{N} \int_{0}^{t^{\prime}} \abs{\vu^{\prime}_{i}(\tau)} \, d\tau
		\\
		&= \frac{1}{N} \sum_{i = 1}^{N} \int_{0}^{t^{\prime}} \abs{\alpha^{\prime} k \left(r_{C} \abs{\vq_{il}(\tau)} \right) \frac{\vq_{il}^{\prime}(\tau)}{\abs{\vq_{il}^{\prime}(\tau)}}
			+ \beta^{\prime} p \left(v_{char} \abs{\vv_{il}^{\prime}(\tau)} \right) \frac{\vv_{il}^{\prime}(\tau)}{\abs{\vv_{il}^{\prime}(\tau)}}} \, d\tau.
	\end{aligned}
\end{equation*}
Here, $\bar{q}_{dev}(t^{\prime})$ and $\bar{v}_{dev}(t^{\prime})$ represent the average deviation of an agent's position and velocity, respectively, from the those of the virtual leader at time $t^{\prime} \ge 0$.
Assuming each agent is propelled by a DC motor connected to a battery, $\bar{U}(t^{\prime})$ represents the total battery drain (see Appendix \ref{appendix:propeller} for details) at time $t^{\prime} \ge 0$, averaged over all agents in the group.
The results of simulations for the dissipative and the non-dissipative scenarios are shown in Figures \ref{comp:fig:cv} and \ref{comp:fig:rot}, respectively.

\begin{figure}[h]
	\centering
	\begin{subfigure}[t]{0.3\linewidth}
		\includegraphics[width=1\textwidth]{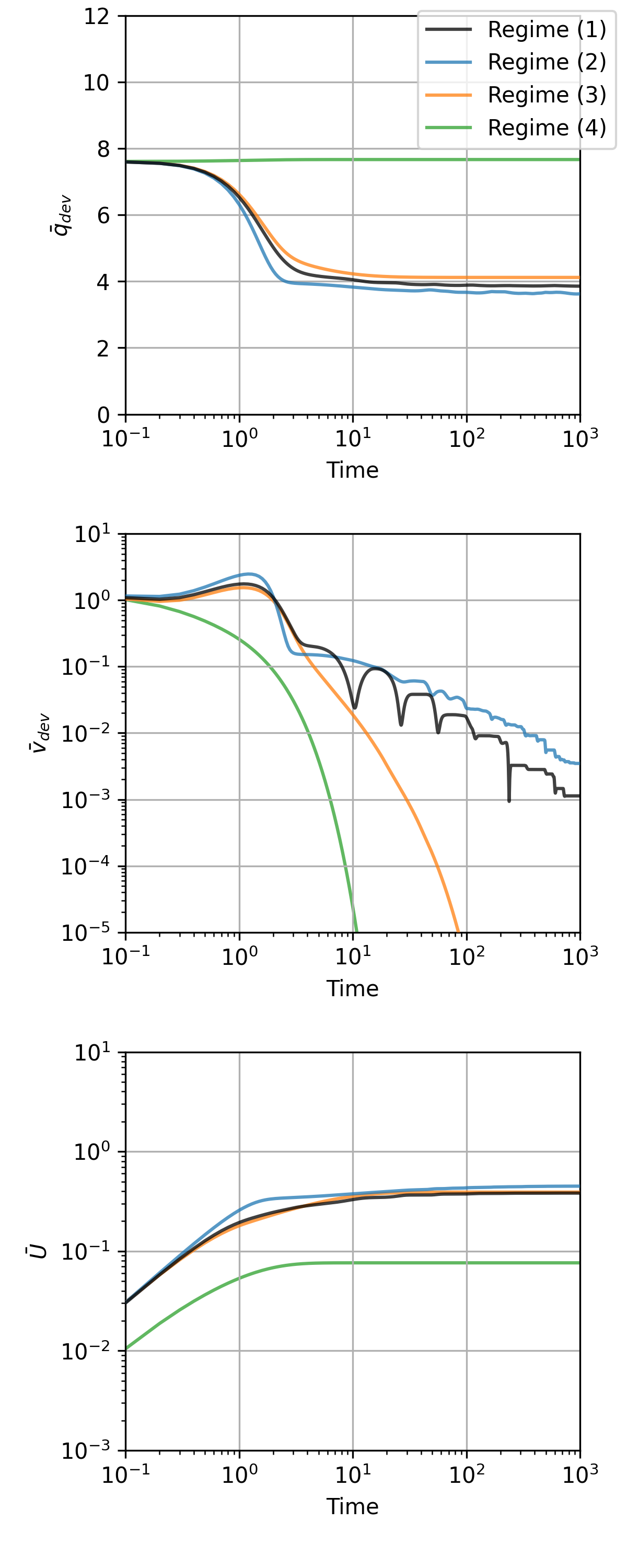}
		\caption{Regimes (1)-(4)}
	\end{subfigure}
	\begin{subfigure}[t]{0.3\linewidth}
		\includegraphics[width=1\textwidth]{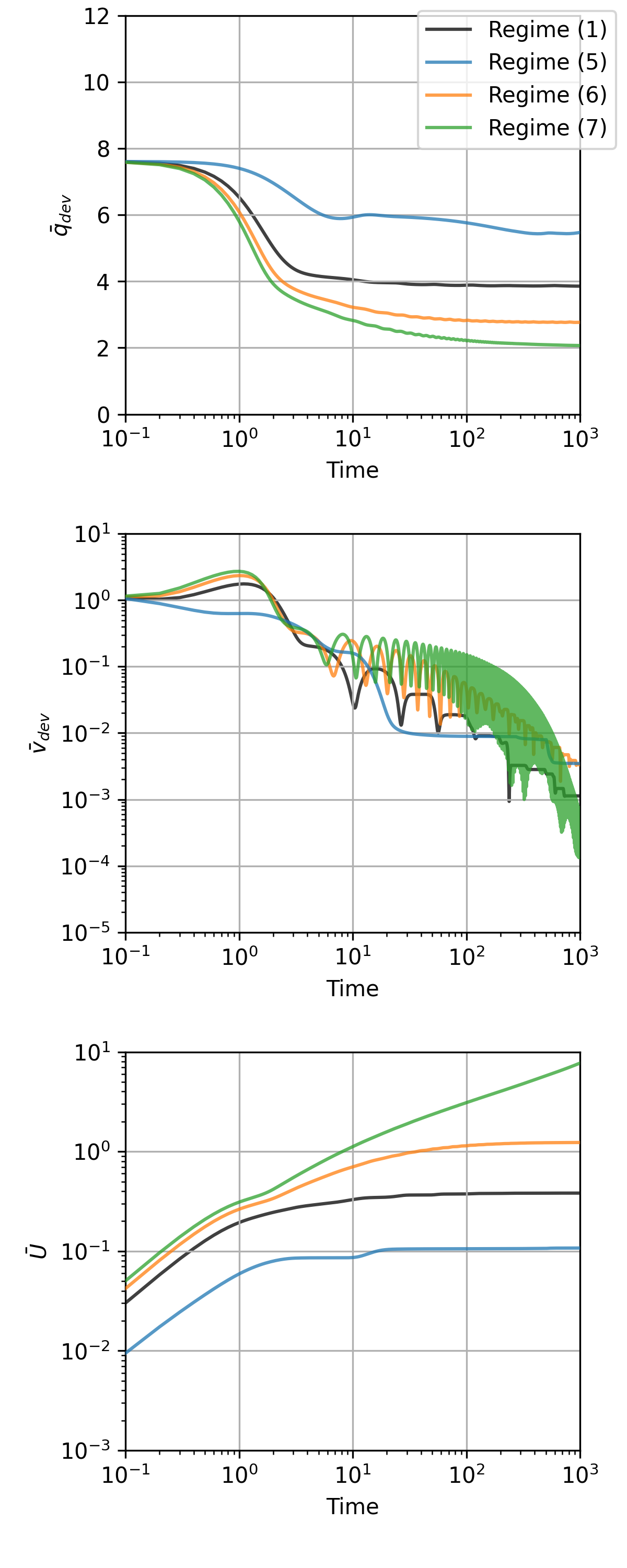}
		\caption{Regimes (1), (5)-(7)}
	\end{subfigure}
	\begin{subfigure}[t]{0.3\linewidth}
		\includegraphics[width=1\textwidth]{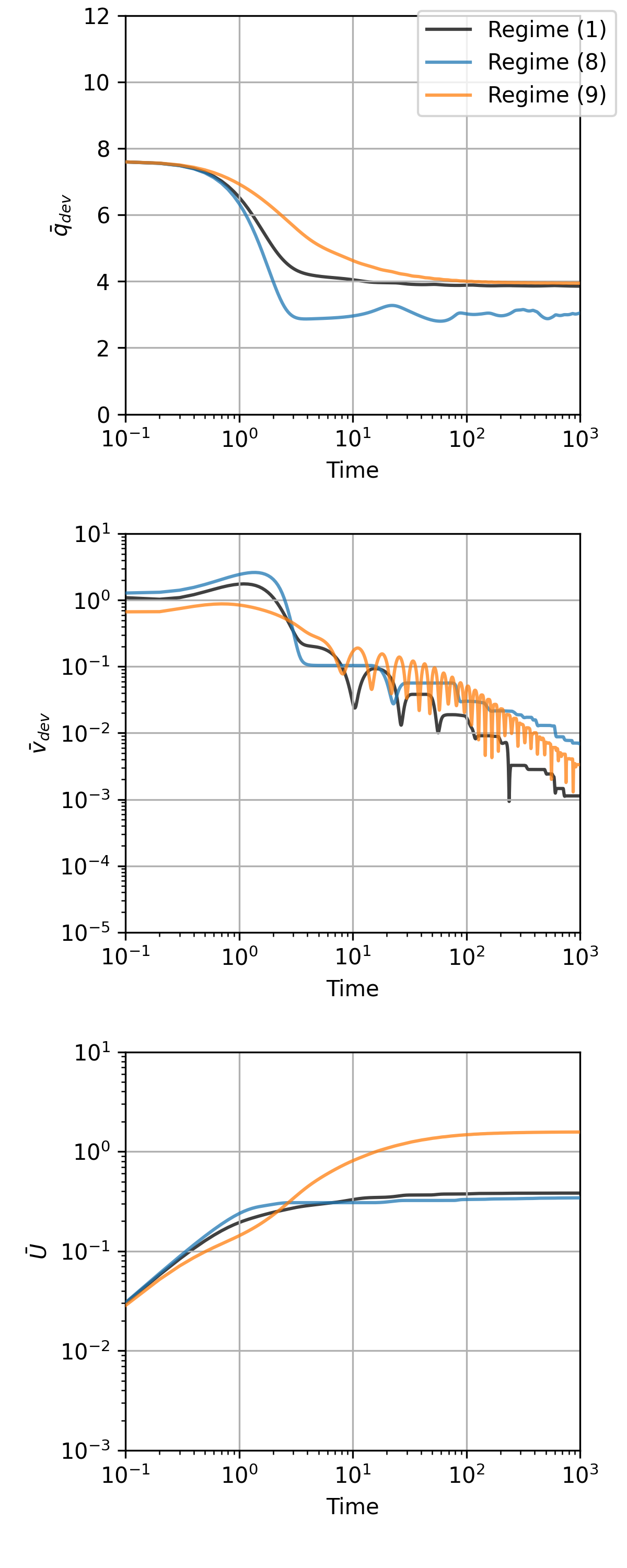}
		\caption{Regimes (1), (8)-(9)}
	\end{subfigure}
	\caption{Plots of $\bar{q}_{dev}(t^{\prime})$, $\bar{v}_{dev}(t^{\prime})$ and $\bar{U}(t^{\prime})$ for the dissipative scenario for a group of $N=100$ agents.}
	\label{comp:fig:cv}
\end{figure}

\begin{figure}[h!]
	\centering
	\begin{subfigure}[t]{0.3\linewidth}
		\includegraphics[width=1\textwidth]{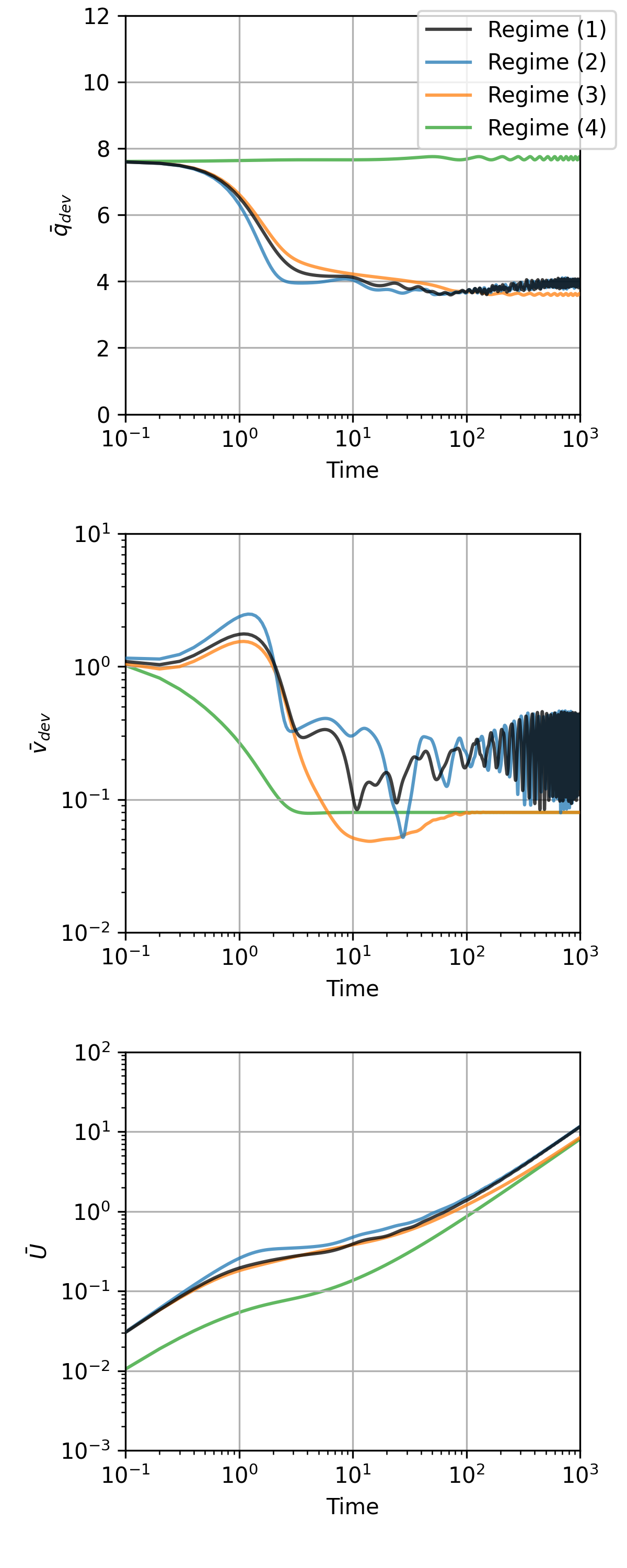}
		\caption{Regimes (1)-(4)}
	\end{subfigure}
	\begin{subfigure}[t]{0.3\linewidth}
		\includegraphics[width=1\textwidth]{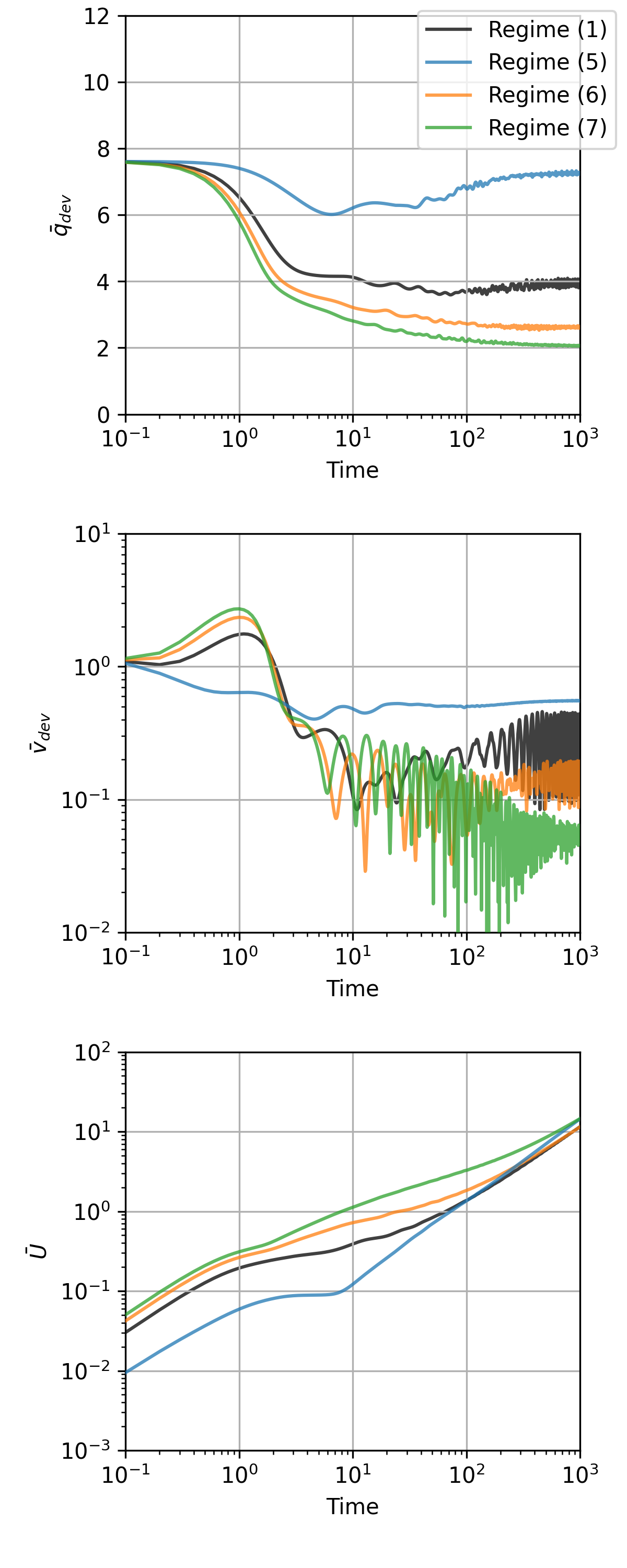}
		\caption{Regimes (1), (5)-(7)}
	\end{subfigure}
	\begin{subfigure}[t]{0.3\linewidth}
		\includegraphics[width=1\textwidth]{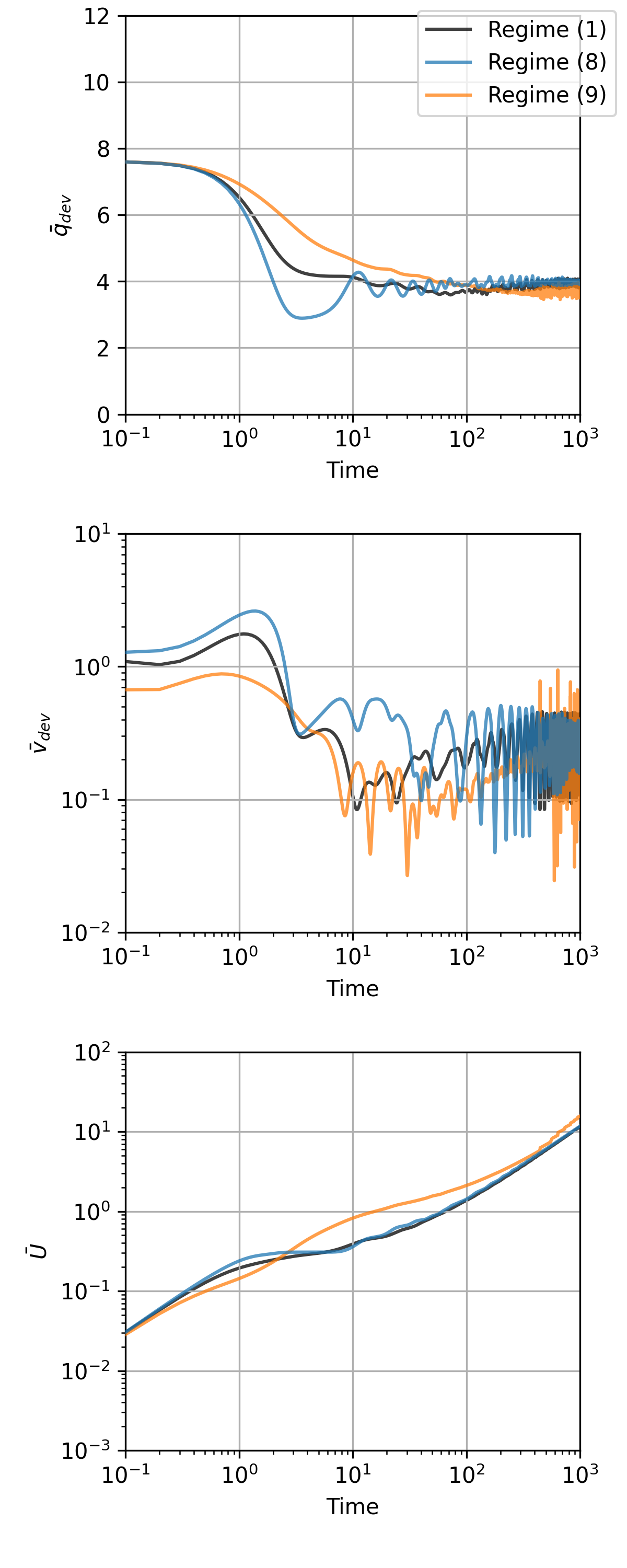}
		\caption{Regimes (1), (8)-(9)}
	\end{subfigure}
	\caption{Plots of $\bar{q}_{dev}(t^{\prime})$, $\bar{v}_{dev}(t^{\prime})$ and $\bar{U}(t^{\prime})$ for the non-dissipative scenario for a group of $N=100$ agents.}
	\label{comp:fig:rot}
\end{figure}

It can be seen from the figures that in all of the regimes, the relative positions of the agents remain bounded.
In the dissipative scenario, a precise asymptotic velocity consensus is achieved in all of the regimes, indicating that the system exhibits proper exact flocking.
The same is true for the non-dissipative scenario in regimes (3) and (4).
In all other regimes under the non-dissipative scenario, only a bounded velocity mismatch is achieved, making the flocking approximate.
Thus, applying ''tight'' velocity alignment is essential for convergence of the system to a proper exact flock in any of the considered scenarios.

Notably, the absence of the position alignment force is associated with a faster convergence to a velocity consensus, provided that a ``tight'' velocity alignment is applied (see regimes (3) and (4)). In situations when spatial dispersion of the group obtained under these regimes is unacceptable, one can use a balanced combination of the two alignment forces as in regime (1).
From the perspective of the battery drain, regimes where the position alignment force is absent, are more efficient.
A more detailed discussion of this aspect is presented in the next section.

Regimes with high values of the group density $\rho$ tend to exhibit faster convergence to a velocity consensus but are characterized by a more volatile spatial configuration.
The former is a consequence of a stronger damping effect of the ambient dissipative force, while the latter results from the competition between the repulsion of the ambient conservative force and the attraction of the position alignment force.
In turn, higher values of $v_{b}$ lead to higher volatility in both positions and velocities.

\subsection{Wobblers}

To illustrate the results obtained in Section \ref{section:wob}, we perform simulations of the dissipative scenario for a group of $N = 3$ agents moving in regimes (1) and (3) with $r_{0}^{\prime}$ set to zero.
Figure \ref{comp:fig:wob} shows the evolution of the differences of agents' positions as well as the agents' velocities.
\begin{figure}[h]
	\centering
	\begin{subfigure}[t]{0.45\linewidth}
		\includegraphics[width=1\textwidth]{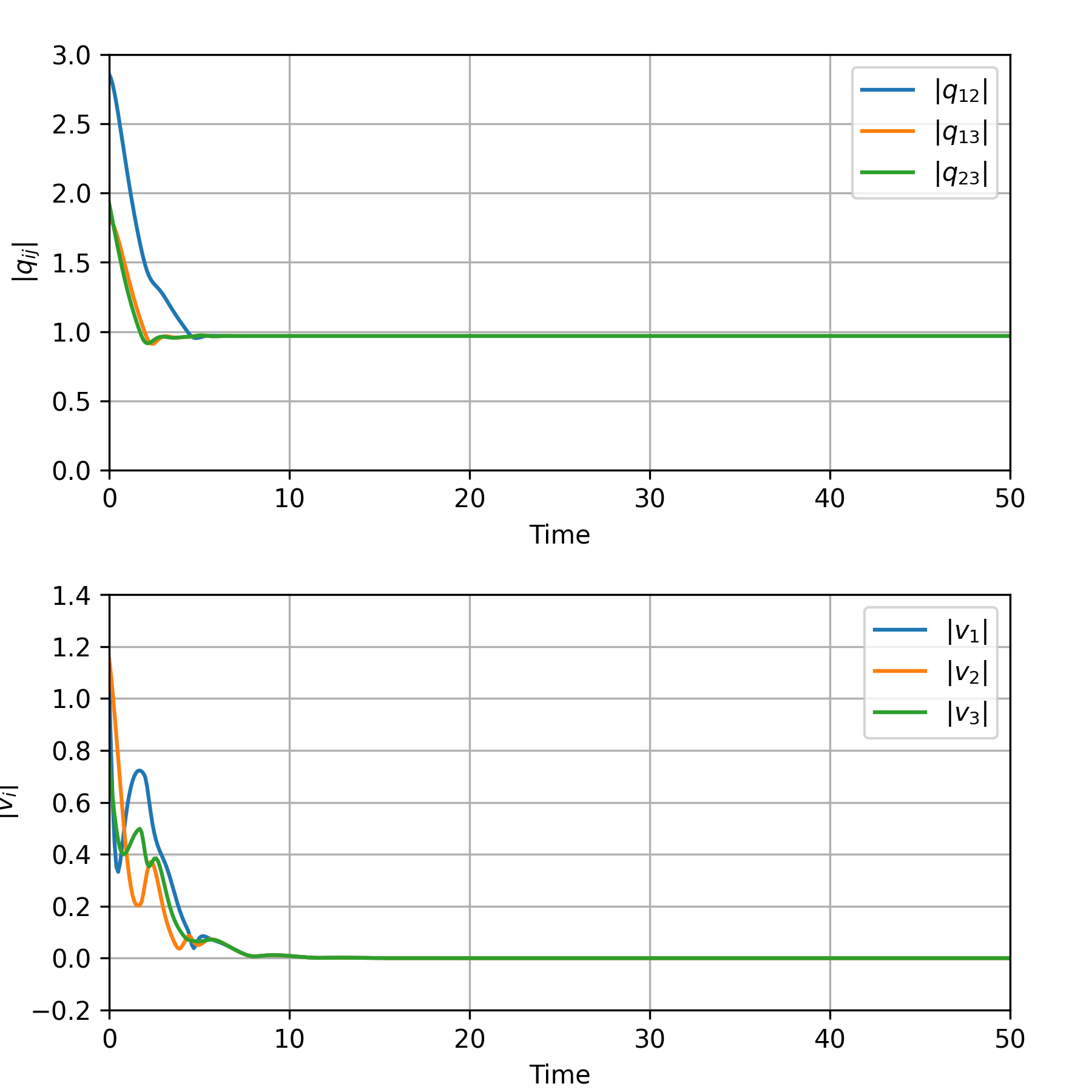}
		\caption{Regime (1)}
	\end{subfigure}
	\begin{subfigure}[t]{0.45\linewidth}
		\includegraphics[width=1\textwidth]{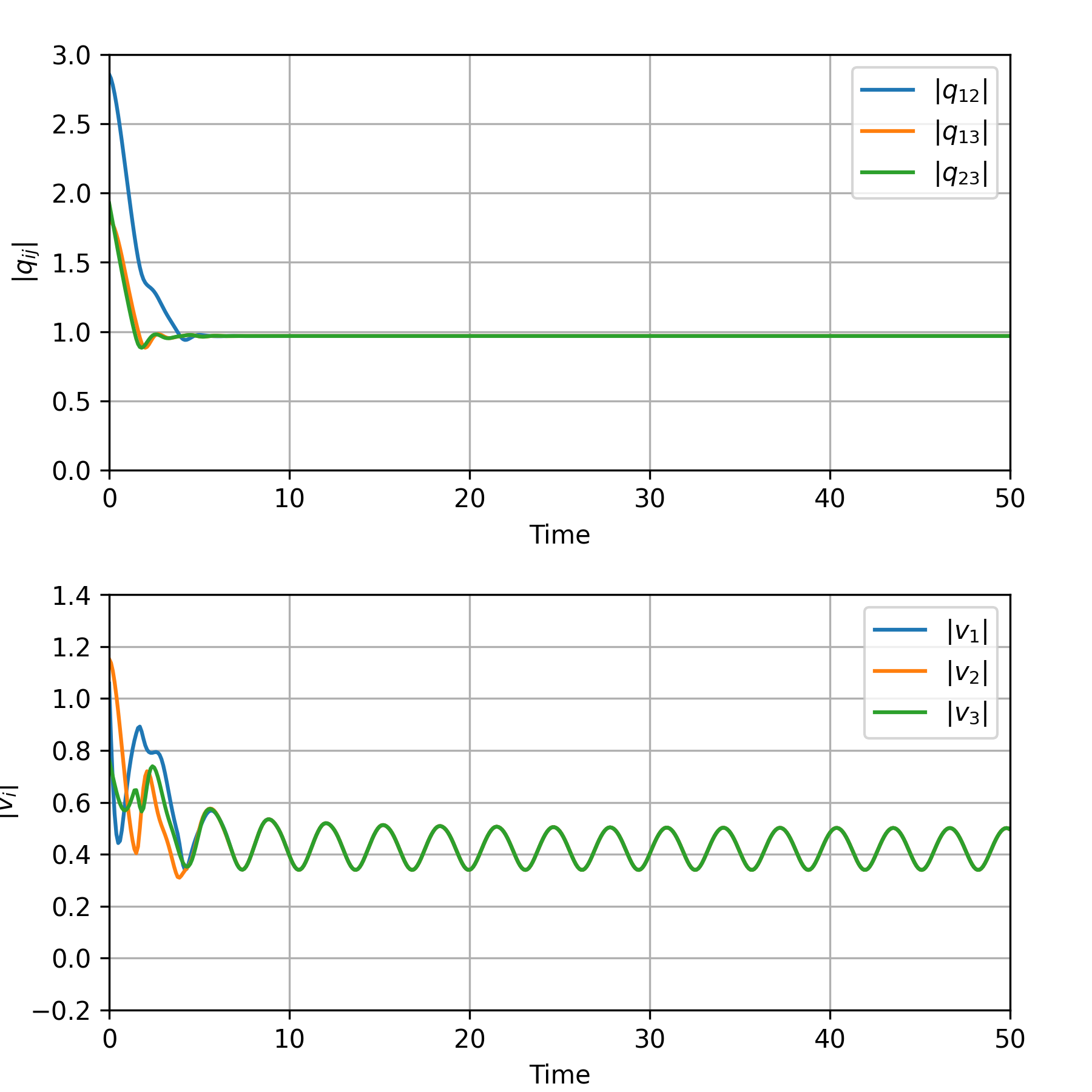}
		\caption{Regime (3)}
	\end{subfigure}
	\caption{Differences of agents' positions and agents' velocities for a group of $N=3$ agents moving in regimes (1) and (3) with $r_{0}=0$. The colored lines represent corresponding measurements for different agents.}
	\label{comp:fig:wob}
\end{figure}
Whereas in both cases the spatial formation of the group converges to a static configuration, the asymptotic behavior of agents' velocities is qualitatively different in the two regimes.
In regime (1),  $v_{0} = 0.5$ so that the explicit periodic solution given by \eqref{wob:eq:linear-periodic} is observed in the velocity plot of Figure \ref{comp:fig:wob} (a).
In turn, regime (3) has $v_{0} = 0$ implying that the system converges to an equilibrium solution (see Theorem \ref{diss:thm:flocking}), which can be observed in Figure \ref{comp:fig:wob} (b).

%% file: sections/07_optimization.tex
\section{Energy-efficient configurations of the self-propulsion forces}
\label{section:opt}

In this section, we analyze the energy efficiency of nonlinear navigational feedback forces. 
We demonstrate how the forces can be fine-tuned to reduce on-board energy consumption. %, which is unattainable using linear forces.
Through numerical simulations, we identify control parameter configurations that minimize the total battery drain of the group while maintaining a desired quality of flock formation. 
We consider the scenario where $k(s) = p(s) = s$, making our navigational feedback forces linear and identical to those of the O-S model when the activation thresholds $r_0$ and $v_0$ are set to zero. 
In this way, we provide a comparative analysis between our navigational feedback forces and the linear ones.
In particular, we show that configurations with $r_{0} = 0$ are suboptimal, illustrating sub-optimality of linear forces used in the O-S model.

Working with the non-dimensional units introduced in the previous section, let $\Gamma = \{ (\vq_{l}^{\prime}(t^{\prime}), \vv_{l}^{\prime}(t^{\prime})), 0 \le t^{\prime} \le T^{\prime}\}$ be a target trajectory that the group must follow to accomplish a certain task, where $T^{\prime} > 0$ is the terminal time of the task.
Let $\vtheta = (\alpha^{\prime}, r_{0}^{\prime}, \beta^{\prime}, v_{0}^{\prime}) \in \Theta$ denote the vector of the control parameters of the self-propulsion forces, where $\Theta \subseteq \mathbb{R}^{4}$ is the set of feasible values. 
Let also $\left(\vq_{i}^{\prime} (t^{\prime}; \vtheta), \v{v}_{i}^{\prime} (t^{\prime}; \vtheta) \right), 0 \le t^{\prime} \le T^{\prime}$, be the solution of \eqref{comp:dynamics-nondim} subject to some initial conditions $(\vQ_{0}^{\prime}, \vV_{0}^{\prime})$ for the target trajectory $\Gamma$ and the parameters configuration $\vtheta$.
Then the average battery drain of the group for the above task is given by
\begin{equation}
	\label{opt:eq:u}
	\bar{U} (\Gamma, \vQ_{0}^{\prime}, \vV_{0}^{\prime}, \vtheta) = \frac{1}{T^{\prime} N} \sum_{i = 1}^{N} \int_{0}^{T^{\prime}} 
	\abs{ 
		\v{u}^{\prime}_{i} \left(\v{q}_{i}^{\prime} (\tau; \vtheta), \vv_{i}^{\prime}(\tau; \vtheta), \v{q}_{l}^{\prime}(\tau), \vv_{l}^{\prime}(\tau); \vtheta \right) 
	} d\tau.
\end{equation}
The quantities
\begin{equation}
	\label{opt:eq:q-v-dev}
	\begin{gathered}
		\bar{Q}_{dev} (\Gamma, \vQ_{0}^{\prime}, \vV_{0}^{\prime}, \vtheta) = \frac{1}{T^{\prime} N} \sum_{i = 1}^{N} \int_{0}^{T^{\prime}} \abs{\v{q}^{\prime}_{i} (\tau; \vtheta) - \vq_{l}^{\prime}(\tau)} d\tau,
		\\
		\bar{V}_{dev} (\Gamma, \vQ_{0}^{\prime}, \vV_{0}^{\prime}, \vtheta) = \frac{1}{T^{\prime} N} \sum_{i = 1}^{N} \int_{0}^{T^{\prime}}  \abs{\vv_{i}^{\prime}(\tau; \vtheta) - \vv_{l}^{\prime}(\tau)} d\tau,
	\end{gathered}
\end{equation}
characterize the overall quality of the group's formation during the task execution by measuring the average deviations of the group from the target position and the target velocity, respectively.
We aim to identify parameter configurations $\vtheta$ that, for given $\Gamma$, $\vQ_{0}^{\prime}$, $\vV_{0}^{\prime}$, yield the smallest possible values of $\bar{U} (\Gamma, \vQ_{0}^{\prime}, \vV_{0}^{\prime}, \vtheta)$ while keeping $\bar{Q}_{dev} (\Gamma, \vQ_{0}^{\prime}, \vV_{0}^{\prime}, \vtheta)$ and $\bar{V}_{dev} (\Gamma, \vQ_{0}^{\prime}, \vV_{0}^{\prime}, \vtheta)$ below reasonable thresholds.

\subsection{Simulations setting}

As in the previous section, we consider a group of $N = 100$ agents in a $3$-dimensional space and use the same parameters of the ambient forces. 
The initial positions $\vQ_{0}^{\prime}$ are drawn from the uniform distribution on the ball $B(\vq_{l}^{\prime}(0), 10)$, and the initial velocities $\vV_{0}^{\prime}$ are drawn from the truncated normal distribution with mean $\vq_{l}^{\prime}(0)$, covariance matrix $0.25 I_{d}$, and support $B(\vq_{l}^{\prime}(0), 0.5)$.
We consider a trajectory $\Gamma$ of the virtual leader that is composed of the following 3 stages: (1) a set of constant acceleration intervals required for the group to move from the initial location to the target location (2) multiple rotations about the target location with constant angular speed (3) a set of constant acceleration intervals required for the group to return to the initial location.
We assume that such a trajectory might be typical in practical applications.
Finally, we choose the feasible set for the control parameters to be
\begin{equation}
	\label{opt:eq:region}
	\Theta = \{ 
	\left(
	\alpha^{\prime}, r_{0}^{\prime}, \beta^{\prime}, v_{0}^{\prime} \right) \in \R^{4} 
	\mid 0 \le \alpha^{\prime} \le 5, \ 0 \le  r_{0}^{\prime} \le 20, \ 0 \le \beta^{\prime} \le 5, \ 0 \le v_{0} \le 2 \}.
\end{equation}

To capture general trends in the effects of control parameter adjustments, we consider a grid $\Xi$ of $26 \times 26 \times 26 \times 26$ equally spaced points on $\Theta$.
For each point in $\Xi$, we perform a simulation using the specified target trajectory and initial conditions, measuring \eqref{opt:eq:u} and \eqref{opt:eq:q-v-dev}.
Figure \ref{opt:fig:grid-1d} shows the values of $\bar{U}$, $\bar{Q}_{dev}$, and $\bar{V}_{dev}$ averaged across all but the $k$-th component of $\vtheta$ for $k = 1, 2, 3, 4$. 
Similarly, Figure \ref{opt:fig:grid-2d} shows the values of $\bar{U}$, $\bar{Q}_{dev}$, and $\bar{V}_{dev}$ averaged across all pairs of the components of $\vtheta$.

\begin{figure}[h]
	\includegraphics[width=1\textwidth]{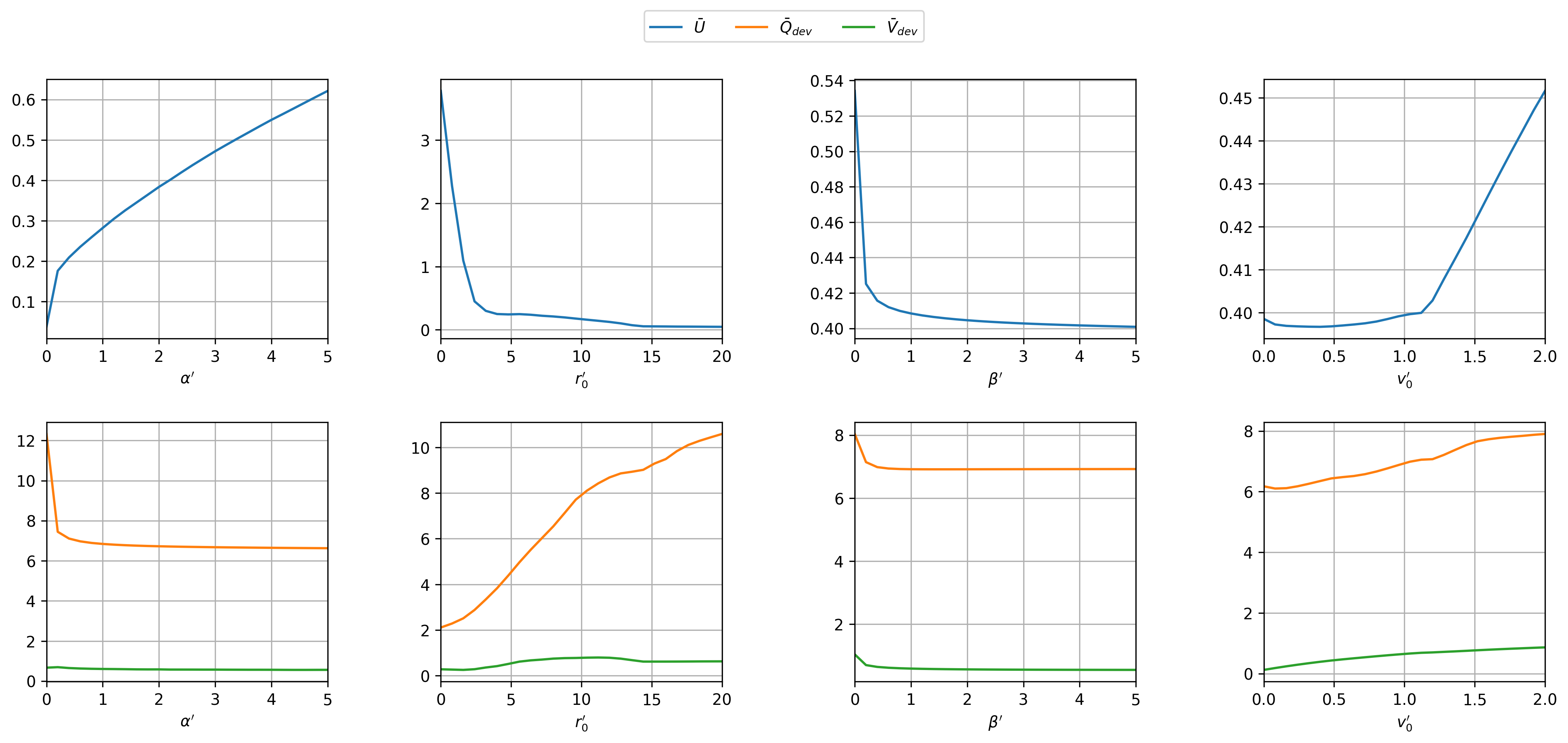}
	\caption{Values of $\bar{U}$, $\bar{Q}_{dev}$, and $\bar{V}_{dev}$ evaluated on $\Xi$, averaged across all but the $k$-th component of $\vtheta$ for $k = 1, 2, 3, 4$.}
	\label{opt:fig:grid-1d}
\end{figure}

\begin{figure}[h!]
	\includegraphics[width=1\textwidth]{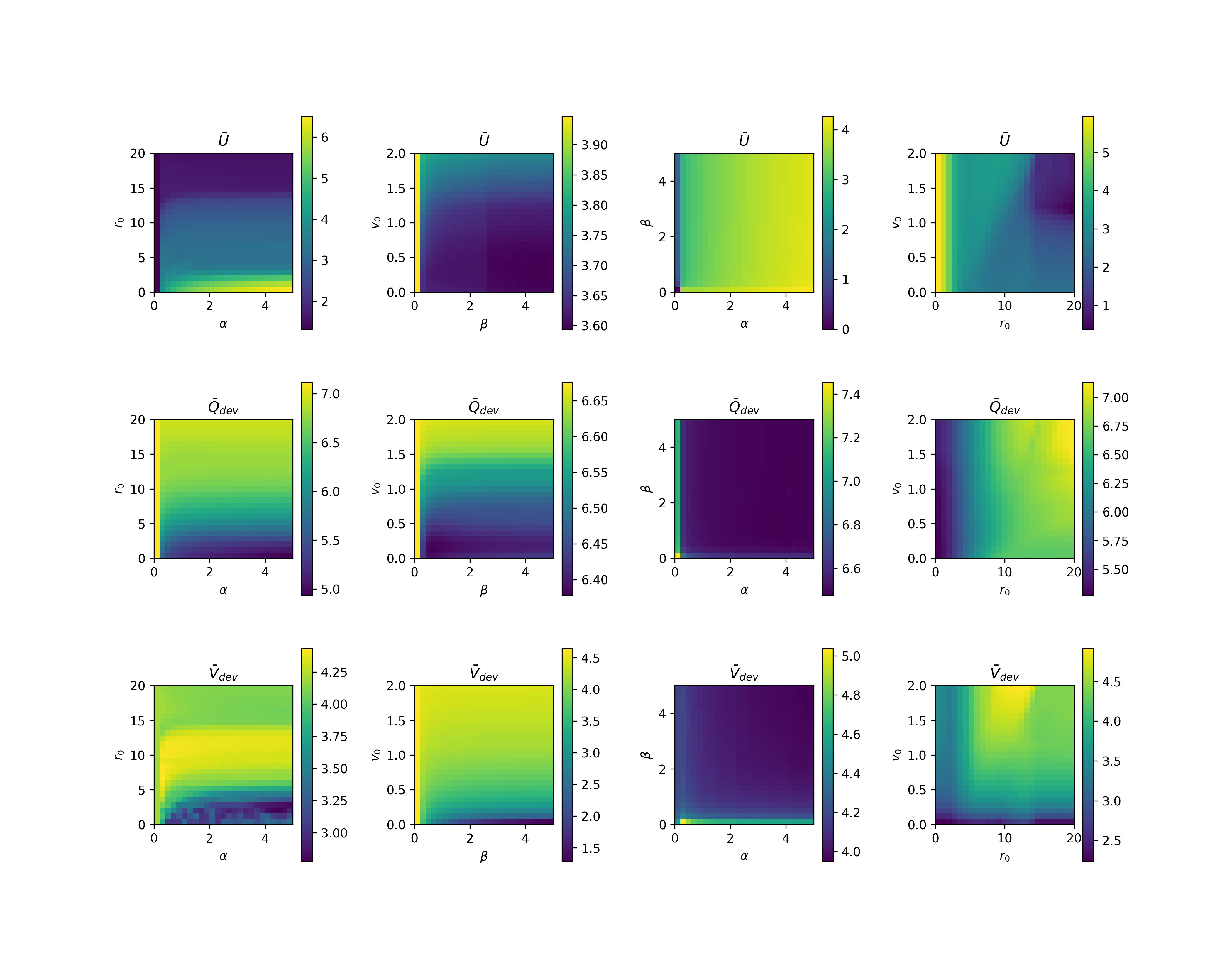}
	\caption{Values of $\bar{U}$, $\bar{Q}_{dev}$, and $\bar{V}_{dev}$ evaluated on $\Xi$, averaged across all all pairs of the components of $\vtheta$.}
	\label{opt:fig:grid-2d}
\end{figure}

\subsection{Discussion of the results}

Figure \ref{opt:fig:grid-1d} reveals that, on average, adjustments to the magnitudes of the position and the velocity alignment forces have opposite effects on the battery drain: increasing $\alpha$ leads to an increase in battery drain, whereas increasing $\beta$ results in a decrease in battery drain.
However, for $\alpha = 0$, the trend for $\beta$ is inverted, and $\bar{U}$ becomes an increasing function of $\beta$ (see Figure \ref{opt:fig:grid-2d}).
Since configurations with $\alpha = 0$ lead to potentially unacceptable spatial divergence, such scenarios are not of interest to us.
Also, amplifying $\beta$ has a notable effect on both $\bar{Q}_{dev}$ and $\bar{V}_{dev}$, whereas amplifying $\alpha$ mainly affects $\bar{Q}_{dev}$.
Therefore, in a typical scenario, it is preferable to perform alignment of small deviations from the target trajectory using the velocity alignment force.
This will reduce the amount of work done by the position alignment force, which should only be activated when strong spatial deviations occur.

Aside from tuning the magnitude of the forces, their effect can be adjusted by configuring the corresponding activation thresholds $r_{0}$ and $v_{0}$.
Decreasing $r_{0}$ enforces a tighter packing of the group and is naturally associated with a battery drain increase.
Figure \ref{opt:fig:grid-1d} shows that a rapid increase in on-board energy consumption occurs when $r_{0}$ falls below approximately 3.5.
Recall that $r_{0}^{\prime} = 3.23$ is the activation threshold that yields the optimal density $\rho = 0.74$.
For the values of $r_{0}^{\prime}$ that are less than this threshold, the position alignment force must be switched on at all times to counteract the action of the ambient conservative force, draining the battery.
Hence, such values are inefficient, and $r_{0}$ should generally be set to the maximum value sufficient to control group spatial dispersion.
Clearly, such a configuration is unattainable for the linear control forces.

Although $\bar{U}$ as a function of the threshold $v_{0}$ attains its minimum at a point distinct from zero, using this value in real-world applications is likely to be inefficient given the steady increase of $\bar{Q}_{dev}$ and $\bar{V}_{dev}$ with increasing $v_{0}$.
Therefore, a larger magnitude and a smaller activation threshold must be used for the velocity alignment force to ensure tight control.

\subsection{Constrained optimization problem}

Determining the actual values of the control parameters that are optimal for a particular task can be approached as either a constrained optimization problem, where one finds the minimum of $\bar{U}$ imposing upper bounds on $\bar{Q}_{dev}$ and $\bar{V}_{dev}$, or a as multi-objective optimization problem, where all the three quantities must be minimized simultaneously.
Taking the former approach, we consider the problem
\begin{eqnarray}
	\label{opt:eq:cosntr-opt-problem}
	\min_{\vtheta \in \Theta} & \bar{U} (\Gamma, \vQ_{0}^{\prime}, \vV_{0}^{\prime}, \vtheta) 
	\nonumber
	\\
	\text{s.t.} & \bar{Q}_{dev} (\Gamma, \vQ_{0}^{\prime}, \vV_{0}^{\prime}, \vtheta) \le Q_{max}
	\\
			    & \bar{V}_{dev} (\Gamma, \vQ_{0}^{\prime}, \vV_{0}^{\prime}, \vtheta) \le V_{max},
	\nonumber
\end{eqnarray}
where $Q_{max}, V_{max} \ge 0$. 

\begin{figure}[h]
	\centering	
	\begin{subfigure}[t]{0.45\linewidth}
		\includegraphics[width=1\textwidth, trim={2.5cm 0 2.5cm 0},clip]{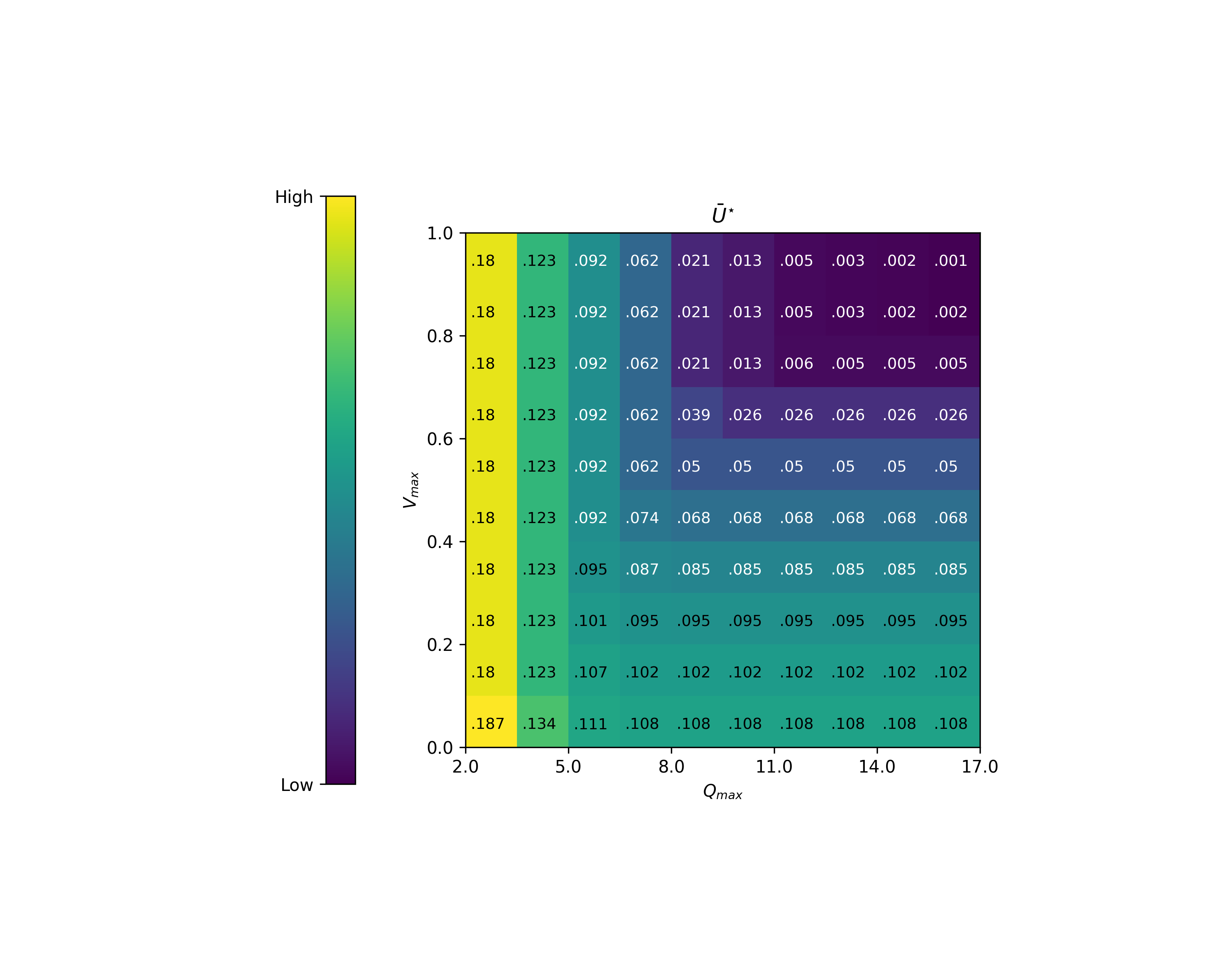}
		\caption{Minima of $\bar{U}$}
	\end{subfigure}
	\begin{subfigure}[t]{0.45\linewidth}
		\includegraphics[width=1\textwidth, trim={2.5cm 0 2.5cm 0},clip]{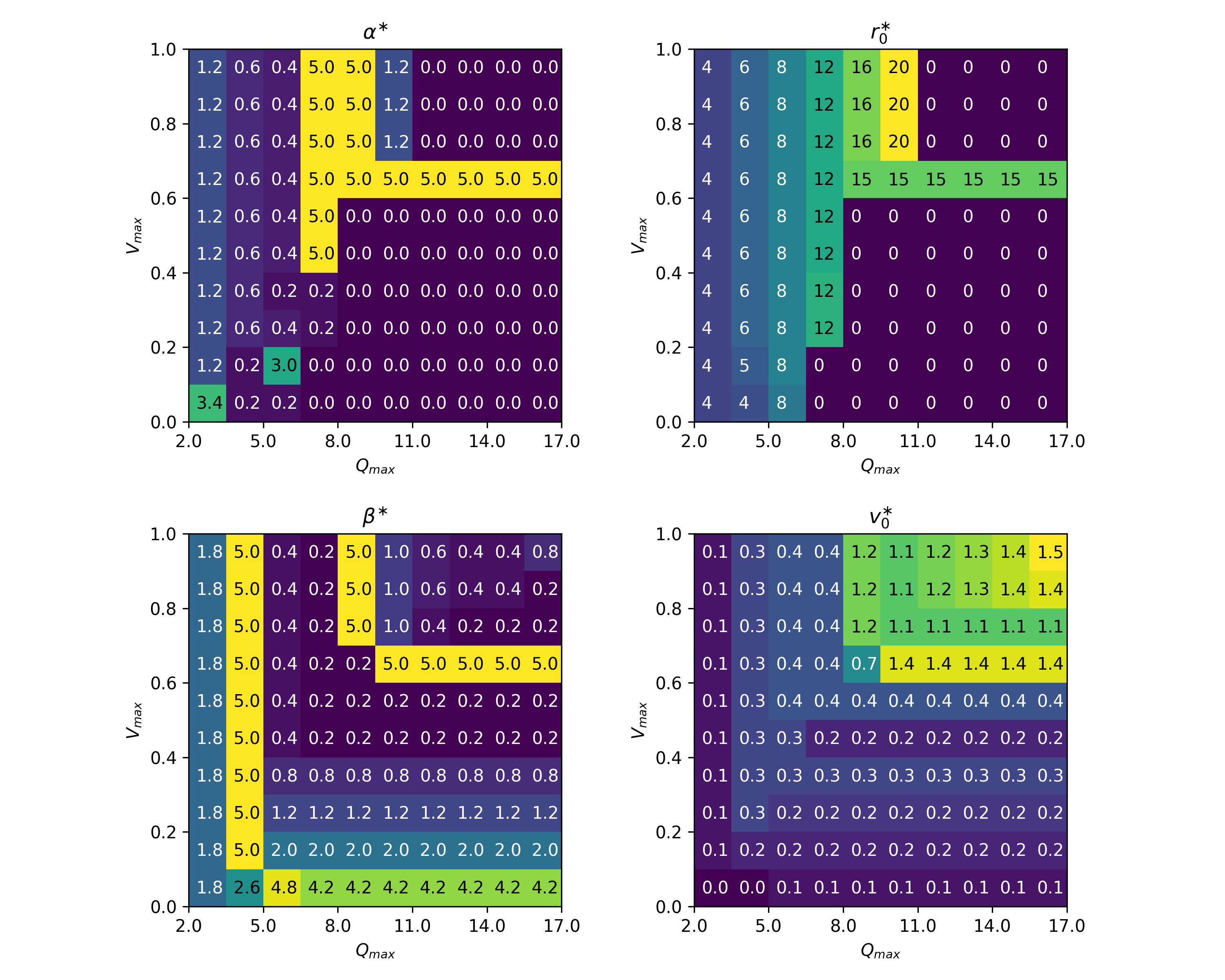}
		\caption{Minimizers of $\bar{U}$}
	\end{subfigure}
	\caption{Solutions of the problem \eqref{opt:eq:cosntr-opt-problem} with $\Theta = \Xi$}
	\label{opt:fig:constrained}
\end{figure}

Putting $\Theta = \Xi$ in \eqref{opt:eq:cosntr-opt-problem}, we solve the problem numerically using the collected data. 
Figure \ref{opt:fig:constrained} shows the solutions for various $Q_{max}$ and $V_{max}$ values.

The absence of monotonic trends for $\alpha$ and $\beta$ seen in Figure \ref{opt:fig:constrained} (b) demonstrates that the true geometry of $\bar{U}(\vtheta)$ as a surface in $\mathbb{R}^{4}$ is more complex than it is suggested by Figure \ref{opt:fig:grid-1d}, where some geometrical features are smoothed out by averaging.

The strictest enforcement of the flock formation is associated with the optimal parameter configuration given by 
\begin{equation*}
	\alpha^{\prime} = 3.4,
	\
	r_{0}^{\prime} = 4,
	\
	\beta^{\prime} = 1.8,
	\
	v_{0}^{\prime} = 0.
\end{equation*}
Relaxing constraints on $Q_{max}$ and $V_{max}$ results in an increase in both the intensity and activation thresholds of the forces. 
The individual importance of the position and the velocity alignment forces in group alignment depends on whether the upper limit on $Q_{max}$ is comparatively stronger than that on $V_{max}$, or vice versa. 
Simultaneously, for certain high values of $Q_{max}$ and $V_{max}$, the position alignment force can become negligible, while the velocity alignment force remains active.
This observation aligns with our previous suggestion emphasizing the relative importance of the velocity alignment force.

%% file: sections/08_conclusion.tex
\section{Conclusion}

We investigated a model of collective motion of a swarm of miniature robots that are controlled by nonlinear navigational feedback virtual forces.
While the physical nature of the governing forces in our model is different, mathematically, it can be seen as a generalization of the model of Olfati--Saber \cite{olfati2006flocking}. 
Although, in general, our system cannot be guaranteed to be dissipative, we showed that, for navigational feedback forces that are bounded perturbations of linear ones, it possesses a global attractor that is characterized by flocking dynamics with bounded deviations of agents' trajectories from the trajectory of the virtual leader.
We obtained explicit bounds on these deviations and showed how these bounds can be controlled by the tunable parameters of the navigational feedback forces.
We also established conditions under which the velocities of all agents converge to the velocity of the center of mass of the group at an exponential rate.

When the system is dissipative, we showed that the attractor can contain non-equilibrium solutions.
We constructed examples of such solutions and obtained some sufficient conditions on the navigational feedback forces under which such solutions cannot exist.

The theoretical findings are supported by numerical simulations.
We also provided a case study of the energy efficiency of the collective dynamics in our model and identified configurations of the control parameters that minimize the on-board energy consumption.

%% file: sections/appendix_dense_lemma_proof.tex
\section{Proof of Lemma \ref{wob:lemma:dense}}
\label{appendix:dense_lemma_proof}

\begin{proof}
Fix an open interval $I \subseteq \R$. 
Suppose that $I \cap S$ is dense in $I$, but there exists $x_{0} \in I$ such that $f(x_{0}) = \gamma > 0$.
Clearly, $x_{0} \not\in I \cap S$, and therefore $x_{0}$ is a limit point of $I$.
By continuity of $f$, there exists $\delta > 0$ such that $\abs{f(x) - f(x_0)} = \abs{f(x) - \gamma} < \gamma$ for all $x \in (x_{0} - \delta, x_{0} + \delta)$. 
Then $-\gamma < f(x) - \gamma$, and hence $f(x) > 0$ on $\tilde{I} = (x_{0} - \delta, x_{0} + \delta) \cap I$.
That is, $\tilde{I}$ does not contain any elements of $I \cap S$, which is a contradiction.
\end{proof}

%% file: sections/appendix_propeller.tex
\section{Battery drain function}
\label{appendix:propeller}
We assume that each agent is set in motion by a propeller that is powered by a DC motor connected to a battery. 
In this case, the battery drain is proportional to the input current $I$ of the DC motor. 
In turn, the current is related to the torque $\v{T}$ produced by the motor as 
\begin{equation*}
	\abs{\v{T}} = \frac{1}{k_{T}} I,
\end{equation*}
where $k_{T}$ is the torque constant \cite{hughes2013electric}. 
Finally, the self-propulsion force, created by the thrust of the propeller, is related to the torque as 
\begin{equation*}
	\abs{\v{u}} = \frac{K_{T} \abs{\v{T}}}{K_{Q} D},
\end{equation*}
where $D$ is the propeller's diameter, and $K_{T}$ and $K_{Q}$ are torque coefficient and thrust coefficient, respectively \cite{carlton2012marine}.
Since $D$, $K_{T}$, and $K_{Q}$ are constants determined by the geometry of the propeller, we get 
\begin{equation*}
	\label{opt:eq:drain-prop}
	\text{agent's battery drain} \propto I \propto \abs{ \v{T} } \propto \abs{ \v{u} }.
\end{equation*} 